\documentclass{amsart}

\usepackage{amssymb}
\usepackage[cp1250]{inputenc}
\usepackage{picture}
\usepackage{color}
\usepackage{tikz}
\usepackage{pifont}
\usepackage{tikz}
\usetikzlibrary{matrix}

\usepackage{amsthm}
\usepackage{enumitem}

\newtheorem{theorem}{Theorem}[section]
\newtheorem{lemma}[theorem]{Lemma}
\newtheorem{proposition}[theorem]{Proposition}
\newtheorem{problem}[theorem]{Problem}
\newtheorem{corollary}[theorem]{Corollary}

\newtheorem{remark}[theorem]{Remark}

\theoremstyle{definition}
\newtheorem{definition}[theorem]{Definition}
\newtheorem{example}[theorem]{Example}

\newcommand{\A}{\mathcal A}

\newcommand{\C}{\mathbb C}

\newcommand{\rng}{{\rm rng} \hspace{0.03cm}}
\newcommand{\Cis}{{\rm Cis} \hspace{0.03cm}}
\newcommand{\Fin}{ {\rm \bf Fin} \hspace{0.03cm}}

\newcommand{\Cv}{\mathcal{C}_{v}}

\newcommand{\GN}{\mathbf{G}\hspace{-0.07cm}\mathbf{N}}

\allowdisplaybreaks

\author{Jacek Marchwicki}
\address{
Faculty of Mathematics and Computer Science,
University of Warmia and Mazury in Olsztyn,
S\l oneczna 54,
10-710 Olsztyn,
Poland}
\email {jacek.marchwicki@uwm.edu.pl}

\author{B\l{}a\.zej \.Zmija}
\address{Charles University, Faculty of Mathematics and Physics, Department of Mathematical Analysis, and Department of Algebra, Sokolov\-sk\' a 83, 18600 Praha~8, Czech Republic \newline
and Institute of Mathematics of the Polish Academy of Sciences, ul. \'{S}niadeckich 8, 00-656 Warszawa, Poland}
\email {blazej.zmija@gmail.com}

\title[Bounded ranges of cardinal functions]{Bounded ranges of cardinal functions}

\thanks{The second author was supported by Czech Science Foundation (GA\v CR) grant 21-00420M, Charles University Research Centre program UNCE/SCI/022, and PRIMUS/25/SCI/017.}

\subjclass[2010]{Primary: 40A05; Secondary: 11K31} 
\keywords{purely atomic measure, achievement set, set of subsums, absolutely convergent series}

\begin{document}

\begin{abstract}
Let $\mathbf{x}$ be a (non-empty) sequence of positive real numbers. Its achievement set $\mathcal{\mathbf{x}}$ is the set of all the possible sums of the elements of $\mathbf{x}$. The cardinal function of $\mathbf{x}$ is the function $f:\mathcal{A}(\mathbf{x}) \to \mathbb{N}\cup\{\omega,\mathfrak{c}\}$ that for every $x\in\mathbb{A}(\mathbf{x})$ the value $f(x)$ is equal to the number of ways $x$ is represented as a sum of elements of $\mathbf{x}$.

In this paper we consider possible ranges of cardinal functions of sequences $\mathbf{x}$. We present some general constructions and several criteria that a set has to satisfy in order to be a range of a cardinal function.

We put special attention to the case of sets with maximal element equal to $6$. In this case, in particular, we obtained a full characterisation of sets that are ranges of cardinal functions of interval-filling sequences.

\end{abstract}

\maketitle 
\section{Introduction}
By the achievement set of a sequence $\mathbf{x}=(x_n)$ we mean the set of all subsums of the series $\sum_{n=1}^{\infty}x_n$, that is, 
\begin{align*}
\mathcal{A}(\mathbf{x})=\left\{\ \sum_{n=1}^{\infty}\varepsilon_n x_n \ \Bigg|\  (\varepsilon_n)\in\{0,1\}^\mathbb{N}\ \right\}=\left\{\ \sum_{n\in A}x_n \ \Bigg|\ A\subseteq\mathbb{N} \ \right\}. 
\end{align*}
The cardinal function $f$ for the sequence $\mathbf{x}$ (or equivalently for the achievement set $\mathcal{A}(\mathbf{x})$) counts the number of representations as subsums for each element of the achievement set. That is, $f: \mathcal{A}(\mathbf{x})\rightarrow \{1,2,\ldots,n,\ldots,\omega,\mathfrak{c}\}$ and for any $x\in \mathcal{A}(\mathbf{x})$ we define 
\begin{align*}
f(x):= \# \left\{\ (\varepsilon_n)\in\{0,1\}^\mathbb{N}\ \Bigg|\ \sum_{x_n\neq 0}\varepsilon_n x_n =x \ \right\},
\end{align*}
where the summation is over those $x_n$'s which are non-zero. We will sometimes write $f_{\textbf{x}}$ if it is not clear, which sequence we are referring to.

Note that if $\mathbf{x}=(x_n)$ is eventually zero, then $\mathcal{A}(\mathbf{x})$ is finite and $f(x)<\omega$ for any $x\in \mathcal{A}(\mathbf{x})$. By the boundedness of $f$ we mean that there exists a natural number $N$ such that $f(x)\leq N$ for all $x$. There are two types of unboundedness: first, when $f$ attains infinite value $\omega$ or $\mathfrak{c}$ and the second, when $f$ achieves infinitely many distinct natural numbers.  
We will also use the notation $f_k$ for the cardinal function for $\mathcal{A}((x_n)_{n\geq k})$.  Cardinal functions were introduced in \cite{cardfun}.
 
 The first paper where achievement sets are considered is that of Kakeya \cite{Kakeya}. The author claimed that if $(x_n)\in l_1$ then $\mathcal{A}(x_n)$ is a finite set, a finite union of closed intervals or a Cantor set. However, due to \cite{GN} and \cite{NS0}, we know that there is one more possibility: so called Cantovals. These are compact subsets of $\mathbb{R}$ with dense interiors. Note that first known examples of Cantorvals were given in the papers \cite{F} and \cite{WS}.

The shape of a cardinal function may depend on one of the four classes to which the achievement set belongs. Therefore, throughout the paper we will need the following notation:
\begin{align*}
\mathcal{I} & := \big\{\ \textrm{possible ranges of cardinal functions obtained for the achievements sets being unions of } \\
& \ \ \ \ \textrm{ closed intervals}\ \big\}, \\
\Cv & := \big\{\ \textrm{possible ranges of cardinal functions obtained for the achievements sets being cantorvals}\ \big\}, \\
\mathcal{C} & := \big\{\ \textrm{possible ranges of cardinal functions obtained for the achievements sets being Cantor sets}\ \big\}, \\
\mathcal{F} & := \big\{\ \textrm{possible ranges of cardinal functions obtained for the achievements sets being finite sets}\ \big\}.
\end{align*}
The subset of $\mathcal{I}$ corresponding to those $f$ whose domain (an achievement set) is a single close interval will be denoted by $\mathcal{I}_{1}$.

Let us recall some known facts about cardinal functions. Most of the results come from the paper \cite{cardfun}. For example, it was proved that $\mathcal{F}\subseteq \mathcal{C}$, namely for any finite achievement set its counterpart Cantor set with the same range of cardinal function was constructed. Moreover if $M\in \mathcal{F}$, then $M\cup\{\mathfrak{c}\}\in \mathcal{I}$. 

Let us denote
\begin{align*}
\mathcal{R}:=\mathcal{I}\cup \Cv\cup \mathcal{C} \cup \mathcal{F},
\end{align*}
that is, the set of all the possible ranges of cardinal functions of summable sequences. It is easy to see that $\mathcal{R}\subsetneq \mathcal{P}(\{1,2,\ldots,n,\ldots,\omega,\mathfrak{c}\})\setminus\{\emptyset\}$, where $\mathcal{P}(X)$ is the power set of $X$. Indeed, for any summable sequences $(x_n)$ the maximal and minimal elements of $\mathcal{A}(x_n)$ are obtained uniquely. Thus for each $M\in \mathcal{R}$ we have $1\in M$.
The more complicated question is whether the equality
\begin{align*}
\mathcal{R}=\big\{\ \{1\}\cup M\ \big|\ M\in\mathcal{P}(\{2,\ldots,n,\ldots,\omega,\mathfrak{c}\}) \ \big\}
\end{align*}
holds, since we do not know any nontrivial method of showing that a given set is not in $\mathcal{R}$.

It is also known that $\{1\}\in \mathcal{C}\setminus(\mathcal{I}\cup \Cv)$, that is, the case when each point in achievement set is obtained by a unique representation is reserved for Cantor sets (and for $\mathcal{F}$ in the more trivial sense of finite subsums). Indeed, note that the mapping $\{0,1\}^{\mathbb{N}} \rightarrow \mathcal{A}((x_n))$ defined as $(\varepsilon_n)\mapsto \sum_{n=1}^{\infty}\varepsilon_n x_n$ is continuous on the compact domain. If $\rng (f)$ denotes the range of $f$, then the equality $\rng(f)=\{1\}$ means that $f$ is a bijection and hence a homeomorphism. Thus the image $\mathcal{A}((x_n))$ is homeomorphic to the Cantor space $\{0,1\}^{\mathbb{N}}$. 

There are still some unsolved problems in the area. For instance, the question: \emph{does $\omega\in M$ imply that $\mathfrak{c}\in M$?} remains unanswered. The method of proving that some set belongs to $\mathcal{R}$ are not well developed. From the few of them the most important is that if $M\in \mathcal{R}$ and $L\in \mathcal{F}$, then $M\cdot L\in \mathcal{R}$. In particular, this implies that $M\cup 2M = \{1,2\}\cdot M\in \mathcal{R}$, and similarly $M\cup 3M\in \mathcal{R}$,  $M\cup 4M\cup 6M\in \mathcal{R}$ or in general $\bigcup_{j=0}^{p} {p\choose j}M\in \mathcal{R}$ for each $p\in\mathbb{N}$. The problem whether the conditions $M\in \mathcal{R}$ and $L\in \mathcal{R}$ imply that $M\cdot L\in \mathcal{R}$ remains open. 

Probably the simplest set about which we still do not know if it belongs to $\mathcal{R}$ is $\{1,4\}$. Note that the sets $\{1,2\}$,  $\{1,3\}$, $\{1,2,3\}$, $\{1,2,3,4\}$, $\{1,2,4\}$, $\{1,3,4\}$  are all known to be the elements of $\mathcal{R}$.
Another known fact is that if $M\in \mathcal{R}$, then $M\cup\{2\}\in \mathcal{R}$. 


The paper is organised as follows. In Section 2 we present some basic properties. Section 3 is dedicated to answering the questions \emph{does $\{1,4\}\in \mathcal{I}_1$?} and more general \emph{does $\{1,m\}\in \mathcal{I}_1$ for $m\in\mathbb{N}$?} by using the concept of lockers \cite{DJK}. In Section 4 we discuss how the upper bound of cardinal function is connected with Kakeya conditions \cite{MarMis}. Then in Section 5 we introduce a useful tool called Cardinality Invariant Sets which help us to recognize unbounded ranges. Section 6 is dedicated to distinguish $\mathcal{I}$ from $\mathcal{I}_1$ by showing that $\mathcal{I}\setminus \mathcal{I}_1$ is nonempty. In Section 7 we obtain more results about $\mathcal{I}_1$, which allow us to complete the characterization of ranges bounded by $6$. Sections 8 and 9 are dedicated to studying  $\mathcal{F}$, $\mathcal{C}$ and  $\Cv$ respectively. In Section 10 we briefly summarize the results by presenting a table showing which sets $M$ with $\max M\leq 6$ belong to sets $\mathcal{I}$, $\Cv$, $\mathcal{\C}$, $\mathcal{F}$ and $\mathcal{R}$. Filling this table is the reason, why we write down many corollaries specifically for sets $M$ with $\max M\leq 6$.

\section*{Convention}

By $\mathbb{N}$ we mean the set of positive integers.

The symbol $(x_{n})$ always means $(x_{n})_{n=1}^{\infty}$. If the set if indices is different than $\mathbb{N}$, we write it explicitly.

In the rest of the paper we consider sequences of positive terms since by \cite[Section 2]{cardfun} we know that the ranges of cardinal functions for $(x_n)$ and $(\vert x_n\vert)$ are the same (up to a shift of the domain which does not affect their ranges).

Moreover, we assume all the sequences appearing in the paper to be summable. In particular, the order of terms of a considered sequence does not affect the shape of its achievement set and cardinal function. Therefore, we can further assume, that all the sequences $(x_{n})$ are nonincreasing:
\begin{align}\label{IneqAssumption}
x_{1}\geq x_{2} \geq x_{3} \geq \ldots > 0.
\end{align}

We will use the above assumptions frequently in the paper without further reference, except when mentioned differently.

\section{General properties}

In this section, we study some preliminary properties of sets $\mathcal{F}$, $\mathcal{I}$, $\mathcal{C}$ and $\Cv$. Recall, that we consider (without loss of generality) only sequences $(x_{n})$ with nonnegative terms.

At first, let us introduce the following notation. For $\mathcal{X},\mathcal{Y}\in\mathcal{P}(\mathbb{R})$ we define
\begin{align*}
\mathcal{X}\cdot \mathcal{Y} := \big\{\ X\cdot Y\ \big|\ X\in\mathcal{X},\ Y\in\mathcal{Y} \ \big\},
\end{align*}
where
\begin{align*}
X\cdot Y:= \big\{\ xy\ \big|\ x\in X,\ y\in Y\ \big\}.
\end{align*}

\begin{lemma}\label{dolozskonczony}
The following equalities are true: 
\begin{align*}
\mathcal{F}\cdot \mathcal{F}=\mathcal{F}, \hspace{1cm} \mathcal{F}\cdot \mathcal{C}=\mathcal{C}, \hspace{1cm} \mathcal{F}\cdot \Cv=\Cv, \hspace{1cm} \mathcal{F}\cdot \mathcal{I}=\mathcal{I}.
\end{align*}
\end{lemma}
\begin{proof}
Since $\{1\}\in \mathcal{F}$, the inclusions '$\supseteq$' in all four equalities are obvious. 
Let $A\in \mathcal{R}$, $B\in \mathcal{F}$. 
Let $\textbf{x}=(x_n)$ be a sequence (finite or not) with the sum $x=\sum x_n$ and the cardinal function $f_{\textbf{x}}$ such that $\rng(f_{\textbf{x}})=A$. Let $\textbf{y}=(y_n)_{n=1}^{k}$ be a finite sequence with the cardinal function $f_{\textbf{y}}$ satisfying $\rng(f_{\textbf{y}})=B$. Let us write $\mathcal{A}(\textbf{y})=\{\sigma_0<\sigma_1<\ldots<\sigma_p\}$ and let $\delta:=\min\{\sigma_i-\sigma_{i-1} \mid i\in\{1,\ldots,p\}\}$.  Then one can find $a\in\mathbb{R}$ such that $a\delta > x$. Let $\tau_{i}:=a\sigma_{i}$. Then $\mathcal{A}(a\textbf{y})=\{\tau_0<\tau_1<\ldots<\tau_p\}$ and  $\min\{\tau_i-\tau_{i-1} \mid i\in\{1,\ldots,p\}\}=a\delta$. Multiplication by a non-zero constant does not change the range of cardinal function so $\rng(f_{a\textbf{y}})=B$.

Let $z_n=ay_n$ for $n\leq k$  and $z_{n+k}=x_n$ for every $n\in\mathbb{N}$. In other words, $\textbf{z}:=(z_n)=(ay_n)\cup(x_n)$. Then $\mathcal{A}(\textbf{z})=\mathcal{A}(a\textbf{y})+\mathcal{A}(\textbf{x})=\bigcup_{i=0}^{p} (\tau_i + \mathcal{A}(\textbf{x}))$ and $\rng(f_{\textbf{z}})=\rng(f_{a\textbf{y}})\cdot \rng(f_{\textbf{x}})=B\cdot A$.

In order to finish the proof let us note that adding finitely many terms to a sequence does not change the form of its achievement set (finite set, finite union of closed intervals, Cantor, Cantorval). Thus $\mathcal{A}(\textbf{z})$ has the same form as $\mathcal{A}(\textbf{x})$, so $B\cdot A$ belongs to the same set among $\mathcal{F}$, $\mathcal{I}$, $\mathcal{C}$, $\Cv$, as $A$.
\end{proof}

\begin{problem}
Is $\mathcal{R}$ closed under multiplication, that is, does $\mathcal{R}\cdot \mathcal{R}= \mathcal{R}$?
\end{problem}

In \cite[Proposition 4.20]{cardfun} the authors considered the sequence $(x_n)$ in which terms do not repeat and proved that it is the only way to obtain $\{1,4\}$ as range of a cardinal function. In the next proposition we consider the opposite case.

\begin{proposition}\label{powtorzonywyraz}
Let $\mathbf{x}=(x_n)$ be such that $\mathcal{A}(\mathbf{x})$ is neither a Cantor set nor a finite set and $x_k=x_{k+1}$ for some $k\in\mathbb{N}$. Then $\max \rng(f)\geq 4$.
\end{proposition}
\begin{proof}
Note that $\mathcal{A}((x_n)_{n>k+1})$ is not a Cantor set. Let $g$ be the cardinal function of $((x_n)_{n>k+1})$. Observe that then there exists $x_{0}$ such that $g(x_{0})\geq 2$. Indeed, if $g(x)=1$ for all $x$ then $\mathcal{A}((x_n)_{n>k+1})$ would be a Cantor set, a contradiction. Hence, $f(x_k+x_{0})\geq 2\cdot g(x_{0})\geq 4$.
\end{proof}

We finish this section with the following fact that is useful when constructing examples of ranges of cardinal functions.

\begin{lemma}
Let $\mathcal{J}\in\{\mathcal{F},\mathcal{I}, \mathcal{I}_{1},\mathcal{C},\Cv\}$. If $M\in \mathcal{J}$ then $M\cup\{2\}\in \mathcal{J}$. 

\end{lemma}
\begin{proof}
Let $\textbf{x}=(x_n)$ be such that its cardinal function $f_{\textbf{x}}$ satisfies $M=\rng(f_{\textbf{x}})\in\mathcal{J}$. If $2\in M$, then we are done. Suppose the opposite and define $y_1=\sum_{n=1}^{\infty}x_n$ and $y_{n+1}=x_n$ for each $n\in\mathbb{N}$. Further, let $\textbf{y}:=(y_{n})_{n=1}^{\infty}$ (in case $\mathcal{J}=\mathcal{F}$ both the sum and the sequence are finite). Clearly $\mathcal{A}(\textbf{y})=\mathcal{A}(\textbf{x})\cup (\sum_{n=1}^{\infty}x_n+\mathcal{A}(\textbf{x}))$. Thus $f_{\textbf{y}}(x)=f_{\textbf{x}}(x)=f_{\textbf{y}}(\sum_{n=1}^{\infty}x_n+x)$ for any $x\in \mathcal{A}(\textbf{x})\setminus\{0,\sum_{n=1}^{\infty}x_n\}$  and $f_{\textbf{y}}(\sum_{n=1}^{\infty}x_n)=2$. Hence $\rng(f_{\textbf{y}})=\rng(f_{\textbf{x}})\cup\{2\}=M\cup\{2\}$. 

To conclude the proof let us observe that for $\mathcal{J}\neq\mathcal{I}_{1}$, then adding to the sequence $\textbf{x}$ one more term (or finitely many terms in general) does not change the form of its achievement set. For $\mathcal{J}=\mathcal{I}_{1}$ note that always $0\in \mathcal{A}(\mathbf{x})$ (it corresponds to the empty sum), so $\mathcal{A}(\mathbf{x})=[0,a]$ for some $a>0$. Then $\mathcal{A}(\mathbf{y})=[0,a]\cup [a,2a]=[0,2a]$, so it belongs to $\mathcal{I}_{1}$. 
\end{proof}

\section{Possible and impossible ranges for interval-filling sequences}

We say that a sequence $\mathbf{x}$ is interval-filling if $\mathcal{A}(\mathbf{x})$ is an interval. In other words, if $f$ is the cardinal function of $\mathbf{x}$ then $\rng(f)\in\mathcal{I}_{1}$. 

One can easily find an interval-filling sequence with $\rng(f)=\{1,2\}$, for instance, take $x_n=\frac{1}{2^n}$. Then $f(x)=2$ if and only if $x$ is a dyadic number and $f(x)=1$ otherwise. In this section we study sets of the form $\{1,m\}$ for $m\geq 3$ as the ranges of cardinal functions.

Note that in the case when $\mathcal{A}(x_n)$ is an interval, if $\frac{1}{2}\sum_{n=1}^{\infty} x_n=\sum_{n\in A} x_n$ for some $A\subseteq\mathbb{N}$, then the equality $\frac{1}{2}\sum_{n=1}^{\infty} x_n=\sum_{n\in\mathbb{N}\setminus A} x_n$ holds as well. It is connected with the property that the achievement set is symmetric with the point of reflection $\frac{1}{2}\sum_{n=1}^{\infty} x_n$. As a result, $f(\frac{1}{2}\sum_{n=1}^{\infty} x_n)$ is either an even number, or is infinite and equal to $\omega$ or $\mathfrak{c}$. In particular, $\{1,m\}\notin \mathcal{I}_1$ for any odd $m$. In general, the bounded range $M\in \mathcal{I}_1$ cannot contain only odd numbers. Therefore, sets $\{1,3,5\}, \{1,3,7\},\{1,5,7\},\{1,3,5,7\}$ etc. are not in $\mathcal{I}_1$. On the other hand, it was shown in \cite[Example 4.19]{cardfun} that $\{1,3,5,7\}\in \mathcal{F}$.

It is worth to mention that \cite[Corolary 2.11]{cardfun} provides a general and simple method of reconstructing any interval-filling sequence into another interval-filling sequence with $\rng(f)=\{1,\mathfrak{c}\}$. Namely, if we begin with an interval-filling sequence $(x_{n})$, then it is enough to repeat each term two times and consider $(x_1,x_1,x_2,x_2,x_3,x_3,\ldots)$. Paper \cite{Nymann} is focused on studying the number of expansions for a point in the interval $[0,2]$ represented by the famous Steinhaus Theorem as the sum of two ternary Cantor sets, that is $[0,2]=C+C=A(2,2;\frac{1}{3})$. The geometric interval-filling sequences $x_n = q^n$ for each $n \in \mathbb{N}$, where $q\geq 1/2$, were studied in paper \cite{Erdos}. The authors set the value $1$ and considered the number of its expansions of the form $1=\sum_{n\in A}q^{n}$ for particular ratios $q$. 

In the following two Lemmas \ref{pologona} and \ref{zwiazekzlockerem}, and Collorary \ref{calewnetrze} we give necessary conditions, which have to be satisfied by any interval-filling sequence with $\rng(f)=\{1,m\}$. Finally, in Theorem \ref{niemozliwy} we show that such sequence does not exist if $m>2$.




In the sequel we will use the following notation.  For a sequence $(x_{n})$, the sequence $(r_{n})$ of sums of its tails is defined for every $n$ as
\begin{align*}
r_{n}:=\sum_{k=n+1}^{\infty}x_{k}.
\end{align*}
The following criterion will be crucial in this part of the paper.

\begin{lemma}\label{LemCritIntervFill}
Sequence $(x_{n})$ is interval-filling if and only if $x_{n}\leq r_{n}$ for all $n$.
\end{lemma}
\begin{proof}
This is a special case of Kakeya's Theorem \cite{Kakeya}. We present the general statement later in Theorem \ref{ThmKakeya}.
\end{proof}

\begin{corollary}\label{CorIntFillA}
If $(x_{n})$ is an interval-filling sequence, then for every $k$ the sequence $(x_{n})_{n>k}$ is also interval-filling. Moreover, we have $\mathcal{A}((x_{n})_{n>k}) = [0,r_{k}]$.
\end{corollary}
\begin{proof}
The fact, that $(x_{n})_{n>k}$ is interval-filling, follows immediately from Lemma \ref{LemCritIntervFill}. Since we assume that $x_{n}\geq 0$ for all $n$, we have $0\leq x\leq r_{k}$ for all $x\in\mathcal{A}((x_{n})_{n>k})$ and both, the lower and the upper bounds, can be obtained. Therefore $\mathcal{A}((x_{n})_{n>k})=[0,r_{k}]$.
\end{proof}

\begin{lemma}\label{pologona}
Let $(x_n)$ be an interval-filling sequence with $\rng(f)=\{1,m\}$ for some $m\geq 3$. Then $x_k\leq\frac{1}{2}r_k$ for all $k\in\mathbb{N}$.
\end{lemma}
\begin{proof}
Fix some $k\in\mathbb{N}$. Then $f(x_k)\geq 2$, since $x_k\in \mathcal{A}((x_n)_{n>k})=[0,r_k]$ by Corollary \ref{CorIntFillA}. Hence $f(x_k)=m$. One can find sets $A_1,A_2,\ldots,A_{m-1}\subseteq \{k+1,k+2,\ldots\}$ such that $x_k=\sum_{n\in A_i}x_n$ for each $i\in\{1,\ldots,m-1\}$. Let
\begin{align*}
x:=x_1+x_2+\cdots+x_{k-1}+2x_k.
\end{align*}
Then for each $i\in\{1,\ldots ,m-1\}$ we have
$$x=\sum_{n\in A_i\cup\{1,\ldots,k\}}x_n.$$
Hence, $f(x)\geq m-1>1$, so $f(x)=m$. One can find a set $D$ different than each of the sets $ A_i\cup\{1,\ldots,k\}$, such that $x=\sum_{n\in D}x_n$. Suppose that $k\in D$. Then for all $i\in\{1,\ldots,m-1\}$ we have the equalities
$$x-x_k=x_1+x_2+\cdots+x_{k-1}+x_k=\sum_{n\in A_i\cup\{1,\ldots,k-1\}}x_n=\sum_{n\in D\setminus\{k\}}x_n,$$
which means that $f(x-x_k)\geq m+1$, a contradiction. Hence, $k\notin D$. Thus $\sum_{n\in D\cup\{k\}}x_n= x_1+x_2+\cdots+x_{k-1}+3x_k$ is an element of achievement set $\mathcal{A}(x_n)$, so cannot be greater than the sum of the series $\sum_{n=1}^{\infty}x_n$. Therefore, we get
$$x_1+x_2+\cdots+x_{k-1}+3x_k\leq \sum_{n=1}^{\infty}x_n,$$
which is equivalent to 
\begin{align*}
2x_{k}\leq \sum_{n=k+1}^{\infty}x_{n}=r_{k}.
\end{align*}
The result follows.
\end{proof}

\begin{lemma}\label{zwiazekzlockerem}
Let $(x_n)$ be an interval-filling sequence such that $x_k\leq\frac{1}{2}r_k$ for all $k\in\mathbb{N}$. Then $(x_n)$ is a locker (that is $x_k\leq r_{k+1}$ for each $k\in\mathbb{N}$, see \cite{DJK} for more details). 
\begin{proof}
Fix $k\in\mathbb{N}$. We have
$$x_k\leq \frac{1}{2}r_k=\frac{1}{2}x_{k+1}+\frac{1}{2}r_{k+1}\leq \frac{1}{2}r_{k+1}+\frac{1}{2}r_{k+1}=r_{k+1}.$$
\end{proof}
\end{lemma}

\begin{corollary}\label{calewnetrze}
Assume that there exists an interval-filling sequence $(x_{n})$ such that $\rng (f)=\{1,m\}$ for some $m\geq 3$. Then
\begin{align*}
f(x)=\left\{\begin{array}{ll}
1, & x\in \left\{0,\sum_{n=1}^{\infty}x_{n}\right\},  \\
 & \\
m, & x\in \left(0,\sum_{n=1}^{\infty}x_{n}\right).
\end{array}\right.
\end{align*}
\end{corollary}
\begin{proof}
In the paper \cite{DJK} it is shown that being a locker suffices to have the smallest possible preimage $f^{-1}(\{1\})$, that is $f^{-1}(\{1\})=\{0, \sum_{n=1}^{\infty}x_n\}$. By Lemmas \ref{pologona} and \ref{zwiazekzlockerem} we know that if there exists an interval-filling sequence with $\rng(f)=\{1,m\}$ for some $m\geq 3$, then it is a locker. The result follows.
\end{proof}

\begin{theorem}\label{niemozliwy}
For any $m\geq 3$ we have $\{1,m\}\notin \mathcal{I}_1$.
\end{theorem}
\begin{proof}
Let $(x_n)$ be interval-filling and $\rng(f)=\{1,m\}$ for some $m\geq 3$. Then by Corollary \ref{calewnetrze} we have $f(x_2)=m$. Hence, there exist $m-1$ different sets $A_1,\ldots,A_{m-1}\subseteq \{3,4,\ldots\}$ such that 
$$
x_2=\sum_{n\in A_1}x_n=\sum_{n\in A_2}x_n=\ldots=\sum_{n\in A_{m-1}}x_n.
$$
By Lemma \ref{pologona} we know that $x_1\leq r_2$, so $x_{1}\in [0,r_{2}]=\mathcal{A}((x_{n})_{n>2})$. Hence, there exists $B\subseteq\{3,4,\ldots\}$ such that $x_1=\sum_{n\in B}x_n$. Thus
$$
x_1+x_2=\sum_{n\in B\cup\{2\}}x_n=\sum_{n\in A_1\cup\{1\}}x_n=\sum_{n\in A_2\cup\{1\}}x_n=\ldots=\sum_{n\in A_{m-1}\cup\{1\}}x_n,
$$
which means that $f(x_1+x_2)\geq m+1$, a contradiction.
\end{proof}


\begin{lemma}\label{conajmniejtrzykrotny}
Let $(x_n)$ be interval-filling and $r_k>x_k$ for some $k\in\mathbb{N}$. Then $\max \rng(f)\geq 3$. 
\begin{proof}
We have $\mathcal{A}((x_n)_{n>k})=[0,r_k]$. Since the subset of all finite subsums of $\sum_{n>k} x_n$ is dense in $\mathcal{A}((x_n)_{n>k})$, one can find a finite set $A\subseteq [k+1,\infty)$ such that $x:=\sum_{n\in A} x_n \in (x_k, r_k)$. Let us write $A=\{n_1<n_2<\ldots<n_i\}$. Then $x= \sum_{n\in A} x_n =\sum_{m=1}^{i} x_{n_m}$. 

Now observe, that Lemma \ref{LemCritIntervFill} and Corollary \ref{CorIntFillA} imply $x_{n_{i}}\in [0,r_{n_{i}}] = \mathcal{A}((x_{n})_{n\geq n_{i}+1})$. Hence, $x_{n_i}=\sum_{n\in B}x_n$ for some $B\subseteq [n_i+1,\infty)$, so $x=  \sum_{m=1}^{i-1} x_{n_m}+\sum_{n\in B}x_n$. Thus $f(x)\geq 2$. Note that $x-x_k\in (0,r_k-x_k)\subseteq [0,r_k]$, so there exists $C\subseteq  [k+1,\infty)$  such that $x-x_k=\sum_{n\in C} x_n$. Thus $x=\sum_{n\in C\cup\{k\}} x_n$. 

To sum up it is enough to observe that all three representations
\begin{align*}
x=\sum_{n\in A} x_n=\sum_{m=1}^{i-1} x_{n_m}+\sum_{n\in B}x_n=\sum_{n\in C\cup\{k\}} x_n
\end{align*}
are different. Indeed, neither $k\notin A$ nor $k\notin (A\setminus\{n_i\})\cup B$, that is, only in the third representation we use index $k$ and $n_i\in A$. Hence $f(x)\geq 3$.
\end{proof}
\end{lemma}

\begin{corollary}\label{CorSet(1,2)}
Let $(x_n)$ be interval-filling. Then $\rng(f)=\{1,2\}$ if and only if  there exists non-zero $c\in\mathbb{R}$ such that $x_n=\frac{c}{2^n}$ for every $n\in\mathbb{N}$. 
\end{corollary} 
\begin{proof}
$\Leftarrow$: is well-known. 

$\Rightarrow$: From Lemmas \ref{LemCritIntervFill} and \ref{conajmniejtrzykrotny} we know that $x_{n}=r_{n}$ for all $n$. Therefore
\begin{align*}
x_{n} = r_{n} = x_{n+1} + r_{n+1} = x_{n+1} + x_{n+1} = 2x_{n+1}.
\end{align*}
Hence, $x_{n+1}=\frac{1}{2}x_{n}$, so one can check using the induction argument, that $x_{n}=\frac{c}{2^{n}}$ for all $n$, where $c>0$ is a constant.
\end{proof}

\section{Kakeya conditions and ranges for interval-filling sequences}

The conditions in which we use the inequalities between the terms $x_n$ and the tails $r_n$ are sometimes called Kakeya conditions, see \cite{MarMis}. They originate from the following celebrated result due to Kakeya.
\begin{theorem}[Kakeya's theorem, \cite{Kakeya}]\label{ThmKakeya}
Let $\mathbf{x}=(x_{n})$. The inequality $x_n\leq r_n$ holds for almost every $n$ if and only if $\mathcal{A}(\mathbf{x})$ is a finite union of compact intervals.

If the inequality $x_n>r_n$ holds for almost every $n$ then $\mathcal{A}(\mathbf{x})$ is a Cantor set.
\end{theorem}
Basically, no other information about achievement set type can be deduced from only Kakeya conditions \cite{MarMis}.

We have seen in Lemma \ref{conajmniejtrzykrotny} that the condition $r_{k}>x_{k}$ for some $k$ implies $\max\rng (f)\geq 3$. In this section we study the connection between the sizes and structures of the sets $\{k\ |\ r_{k}>x_{k}\}$ and $\rng (f)$.

In the paper \cite[Theorem 3.8]{cardfun} it is proved for interval-filling sequences that if there are finitely many strict inequalities $r_n>x_n$, then the cardinal function $f$ is bounded, and if $r_n>x_n$ for infinitely many $n$, then $f$ attains $\omega$ or $\mathfrak{c}$. It is worth to mention that in general it is still not known if unboundedness of the range $M$ can occur with $\{\omega,\mathfrak{c}\}\cap M=\emptyset$ (see \cite[Problems 2.12 and 2.13]{cardfun}). If the answer was positive, then there would exist $n_i\to \infty$ such that $n_i\in M$ for every $i\in\mathbb{N}$. 

Let us prove a lemma and two theorems that will be useful later in this section. They answer the question: what is the cardinal function of the geometric sequence with ratio $\frac{1}{2}$ with exactly one more element $x_{1}$ added. Let $\mathcal{D}$ denote the set of all dyadic numbers in the unit interval. We consider two separate cases depending on whether $x\in\mathcal{D}$ or $x\notin\mathcal{D}$. Let us begin with a basic lemma which clarifies the later results. 

\begin{lemma}\label{lematdodwojkowych}
Let $x_n=\frac{1}{2^n}$ and $f$ be the cardinal function for $(x_n)$. Then 
\begin{displaymath}
f(x) = \left\{ \begin{array}{ll} 2, & \textrm{if $x\in  \mathcal{D}\setminus\{0,1\}$,}\\ 1, & \textrm{if $x\in \big((0,1)\setminus \mathcal{D}\big)\cup\{0,1\}$.} \end{array} \right.
\end{displaymath}
\end{lemma}
\begin{proof}
Follows quickly from the definitions.
\end{proof}

\begin{theorem}\label{dwojkowowymierna}
Let $\mathbf{x}=(x_{n})$ be such that $x_{1}\in\mathcal{D}$, $x_{1}<1$, and $x_{n}=\frac{1}{2^{n-1}}$ for $n\geq 2$, and let $g$ be its cardinal function. Then $\rng (g)=\{1,2,3,4\}$. In particular, $\{1,2,3,4\}\in \mathcal{I}_1$.
\end{theorem}
\begin{proof}
We have $\mathcal{A}(\textbf{x})=[0,1+x_1]$. Fix $x\in[0,1+x_1]$. Then
\begin{align*}
\bigg|\bigg\{ A\subseteq \mathbb{N}\ \bigg|\ \sum_{n\in A} x_n =x \bigg\}\bigg|= \bigg|\bigg\{ A\subseteq \mathbb{N}\setminus\{1\}\ \bigg|\ \sum_{n\in A} x_n =x \bigg\}\bigg| + \bigg|\bigg\{ A\subseteq \mathbb{N}\setminus\{1\}\ \bigg|\ \sum_{n\in A} x_n =x-x_1  \bigg\}\bigg|.
\end{align*}
The above equality can be rewritten in terms of cardinal functions as $g(x)=f(x)+f(x-x_1)$, where $f$ is the cardinal function of the sequence $\left(\frac{1}{2^{n}}\right)_{n=1}^{\infty}$ decribed in Lemma \ref{lematdodwojkowych}. Note that the sum of two dyadic numbers is dyadic. Thus we have
\begin{displaymath}
g(x) = f(x)+f(x-x_1)= \left\{ \begin{array}{ll} 
1+0=1, & \textrm{if $x=0$,}\\ 
2+0=2, & \textrm{if $x\in(0,x_1)\cap \mathcal{D}$,}\\ 
1+0=1, & \textrm{if $x\in(0,x_1)\setminus \mathcal{D}$,} \\ 
2+1=3, & \textrm{if $x=x_1$,}\\ 
2+2=4, & \textrm{if $x\in (x_1,1)\cap \mathcal{D}$,}\\ 
1+1=2, & \textrm{if $x\in (x_1,1)\setminus \mathcal{D}$,}\\ 
1+2=3, & \textrm{if $x=1$,}\\ 
0+2=2, & \textrm{if $x\in (1,1+x_1)\cap (x_1+\mathcal{D})$,}\\ 
0+1=1, & \textrm{if $x\in (1,1+x_1)\setminus (x_1+\mathcal{D})$,}\\ 
0+1=1, & \textrm{if $x=1+x_1$.}\end{array} \right.
\end{displaymath}
Thus indeed $\rng (g)=\{1,2,3,4\}$.
\end{proof}

\begin{theorem}\label{niedwojkowowymierna}
Let $\mathbf{x}=(x_{n})$ be such that $x_{1}\notin\mathcal{D}$, $x_{1}\in (0,1)$, and $x_{n}=\frac{1}{2^{n-1}}$ for $n\geq 2$, and let $h$ be its cardinal function. Then $\rng (h)=\{1,2,3\}$. In particular, $\{1,2,3\}\in \mathcal{I}_1$.
\end{theorem}
\begin{proof}
Similarly as in the proof of Theorem \ref{dwojkowowymierna} we have $h(x)=f(x)+f(x-x_1)$. Hence 
\begin{displaymath}
h(x) = f(x)+f(x-x_1)= \left\{ \begin{array}{ll} 
1+0=1, & \textrm{if $x=0$,}\\ 
2+0=2, & \textrm{if $x\in(0,x_1)\cap \mathcal{D}$,} \\ 
1+0=1, & \textrm{if $x\in(0,x_1)\setminus \mathcal{D}$,} \\ 
1+1=2, & \textrm{if $x=x_1$,}\\ 
2+1=3, & \textrm{if $x\in (x_1,1)\cap \mathcal{D}$,}\\ 
1+2=3, & \textrm{if $x\in (x_1,1)\cap (x_1+\mathcal{D})$,}\\ 
1+1=2, & \textrm{if $x\in (x_1,1)\setminus (\mathcal{D}\cup(x_1+\mathcal{D}))$,}\\ 
1+1=2, & \textrm{if $x=1$,}\\ 
0+2=2, & \textrm{if $x\in (1,1+x_1)\cap (x_1+\mathcal{D})$,}\\ 
0+1=1, & \textrm{if $x\in (1,1+x_1)\setminus (x_1+\mathcal{D})$,}\\ 
0+1=1, & \textrm{if $x=1+x_1$,}\end{array} \right.
\end{displaymath}
and the result follows.
\end{proof}

Now we are going to prove that in fact the achievement sets from Theorems \ref{dwojkowowymierna} and \ref{niedwojkowowymierna} are the only possibilities if we assume that the inequality $r_{k}>x_{k}$ holds for exactly one index $k$. At first, let us prove the following lemma which shows a consequence of the equalities between two consecutive terms and the counterpart tails.

\begin{lemma}\label{dwasasiednie}
Let $x_n=r_n$ and $x_{n+1}=r_{n+1}$ for some $n\in\mathbb{N}$. Then $x_n=2x_{n+1}$.
\begin{proof}
$x_n=r_n=x_{n+1}+r_{n+1}=2x_{n+1}$
\end{proof}
\end{lemma}

The next Corollary shows that if the the cardinal function is bounded then the sequence resembles a geometric sequence. 

\begin{corollary}\label{geometrycznyogon}
Let $(x_n)$ be interval-filling and $M$ be the range of its cardinal function. Then $M$ is bounded if and only if there exist $k\in\mathbb{N}$ and nonzero $c\in\mathbb{R}$ such that $x_n=\frac{c}{2^n}$ for all $n\geq k$.
\end{corollary}
\begin{proof}
It is an immediate consequence of \cite[Theorem 3.8]{cardfun} and Lemma \ref{dwasasiednie}.
\end{proof}

\begin{theorem}\label{musibycdwojka}
Let $(x_n)$ be interval-filling and let $f$ be its cardinal function. If $\rng (f)$ is bounded, then $2\in \rng (f)$. 
\begin{proof}
By Corollary \ref{geometrycznyogon} there exist $k\in\mathbb{N}$ and positive $c\in\mathbb{R}$ such that $x_n=\frac{c}{2^n}$ for all $n\geq k$. Clearly, $\mathcal{A}((x_n)_{n\geq k})=[0,\frac{c}{2^{k-1}}]$. Let $g$ be the cardinal function of $(x_n)_{n\geq k}$. Since $x_m\geq x_k=\frac{c}{2^k}$ for all $m<k$ (recall the assumption \eqref{IneqAssumption}), we have $f(y)=g(y)$ for all $y\in [0,x_k)$. In particular, Lemma \ref{lematdodwojkowych} yields that $f(x)=2$ on every scaled dyadic number $x\in \mathcal{D}\cdot\frac{2^k}{c}$.
\end{proof}
\end{theorem}

Theorem \ref{musibycdwojka} allows us to exclude many sets from $\mathcal{I}_1$. At the end of the paper we present a table summarising which sets $M$ with $\max M\leq 6$ can occur as a ranges of cardinal functions of sequences with presribed type of the achievement set. Therefore, let us write down which sets with the largest element at most $6$ are excluded from $\mathcal{I}_{1}$.

\begin{corollary} \label{jedentrzyczteryprzedzial}
Sets $\{1,3,4\}$, $\{1,4,5\}$, $\{1,3,4,5\}$, $\{1,3,6\}$, $\{1,4,6\}$, $\{1,5,6\}$, $\{1,3,4,6\}$, $\{1,3,5,6\}$, $\{1,4,5,6\}$, $\{1,3,4,5,6\}$ do not belong to $\mathcal{I}_1$.
\end{corollary}

\begin{theorem}\label{jednaostra}
Let $(x_n)$ be an interval-filling sequence such that the strict inequality $r_k>x_k$ holds for exactly one number $k$. Then $\rng(f)$ is equal to either $\{1,2,3\}$ or $\{1,2,3,4\}$.
\begin{proof}
Lemma \ref{LemCritIntervFill} implies that $x_n=r_n$ for all natural numbers but $k$ and $x_k<r_k$. By Lemma \ref{dwasasiednie} the sequence $(x_n)$ looks as follows:
$$\left(2^{k-2}x_{k-1}, 2^{k-3}x_{k-1}, \ldots, 8x_{k-1} , 4x_{k-1} , 2x_{k-1} , x_{k-1}, x_k, x_{k+1}, \frac{x_{k+1}}{2},\frac{x_{k+1}}{4},\frac{x_{k+1}}{8}, \ldots\right). $$
Since the range of cardinal function would not change when scaling the sequence, we can divide all the terms by $2x_{k+1}$. Hence, we get the sequence
$$
\left(\frac{2^{k-3}x_{k-1}}{x_{k+1}},  \frac{2^{k-4}x_{k-1}}{x_{k+1}}, \ldots,  \frac{4x_{k-1}}{x_{k+1}} ,  \frac{2x_{k-1}}{x_{k+1}} ,  \frac{x_{k-1}}{x_{k+1}} , \frac{x_{k-1}}{2x_{k+1}}, \frac{x_k}{2x_{k+1}}, \frac{1}{2}, \frac{1}{4},\frac{1}{8},\frac{1}{16}, \ldots\right) 
$$ 
Denote this new sequence by $(y_n)$. More precisely, $y_n=\frac{x_{n}}{2x_{k+1}}$ for all $n\in\mathbb{N}$. Let us denote the sequence of sums for the tails of $(y_{n})$ by $(t_n)$, that is, $t_{n}:=\sum_{m>n}y_{m}$ for all $n$. Clearly $t_n=\frac{r_{n}}{2x_{k+1}}$ for all $n\in\mathbb{N}$.
Note that $y_k=\frac{x_k}{2x_{k+1}}<\frac{r_k}{2x_{k+1}}=t_k=1$ and $y_k=\frac{x_k}{2x_{k+1}}\geq \frac{x_{k+1}}{2x_{k+1}}=\frac{1}{2}$. Moreover, $t_{k-1}=y_k+t_k=y_k+1$ and $y_{k-1}=t_{k-1}=y_k+1$. Let us put $y:=y_k\in [\frac{1}{2},1)$. Hence, we can rewrite the previous sequence as 
$$
\left(2^{k-2}(y+1), 2^{k-3}(y+1), \ldots,  8(y+1) ,  4(y+1),  2(y+1) ,(y+1), y, \frac{1}{2}, \frac{1}{4},\frac{1}{8},\frac{1}{16}, \ldots\right).
$$
If $k=1$ this becomes $( y, \frac{1}{2}, \frac{1}{4},\frac{1}{8},\frac{1}{16}, \ldots)$ and the statement is an immediate consequence of Theorems \ref{dwojkowowymierna} and \ref{niedwojkowowymierna}.

Assume that $k>1$. Let $g$ and $f$ be the cardinal functions of $(y_n)_{n\geq k}$ and $(y_n)$ respectively.

Note that in the first part of the proof we showed that $\rng(g)$ is either $\{1,2,3\}$ or $\{1,2,3,4\}$. Now we show that $\rng(f)=\rng(g)$. Let $x\in \mathcal{A}(y_n)=\left[0,\sum_{n=1}^{\infty}y_n\right]=[0,2^{k-1}(y+1)]$. Consider the following four cases:
\begin{itemize}
\item If $x=a(y+1)$, where $a\in\{0,2^{k-1}\}$, then clearly $f(x)=1$.
\item If $x\neq a(y+1)$ for any natural $a$ then there exists natural number $c\in [0,2^{k-1}]$ such that $c(y+1)<x<(c+1)(y+1)$. If $x=\sum_{n\in A} y_n$, then the set $A$ satisfies $\emptyset\neq A\cap [k,\infty)\neq  [k,\infty)$, because otherwise $x$ is a natural multiple of $(y+1)$ since the first $k-1$ elements of $(y_n)$ and $\sum_{n\geq k} y_n=y+1$ are multiples of $y+1$. Hence we obtain two parts of $x$, that is $x=\sum_{n\in A} y_n=\sum_{n\in A\cap [1,k-1]} y_n+\sum_{n\in A\cap[k,\infty)} y_n$. The first summand is the natural multiple of $y+1$ while the second is less than $y+1$, which means that $c(y+1)=\sum_{n\in A\cap [1,k-1]} y_n$ or equivalently  $c=\sum_{n\in A\cap [1,k-1]} \frac{y_n}{y+1}$.
The number $c$ has a unique binary expansion, so there exists a unique binary sequence $(\varepsilon_n)_{n=1}^{k-1}$ such that $c(y+1)=\sum_{n=1}^{k-1}\varepsilon_{n}y_n$. Hence $x$ has the same number of representations as $x-c(y+1)$, so $f(x)=f(x-c(y+1))=g(x-c(y+1))$. 
\item If $x=a(y+1)$ for an odd number $a\in (0,2^{k-1})$ then one can easily check that $f(x)=2$.
\item If $x=a(y+1)$ for an even number $a\in (0,2^{k-1})$ then again $f(x)=2$.
\end{itemize}
This shows that $\rng(f)=\rng(g)\cup\{1,2\}=\rng(g)$ and therefore the result follows.
\end{proof}
\end{theorem}

Theorem \ref{jednaostra} characterizes the interval-filling sequences for which only one remainder is strictly greater than the preceeding term. It is natural to ask what happens in the case of more strict inequalities. We can use Theorem \ref{niedwojkowowymierna} to construct sequences $(x_{n})_{n=1}^{\infty}$ with arbitrarily (but finitely) many strict inequalities but very small range of the cardinal functions.

\begin{example}\label{zawszerazdwatrzy}
For every $K$ there exists an interval-filling sequence $(x_{n})_{n=1}^{\infty}$ such that 
$$
\{ n\in\mathbb{N}\ |\ r_n>x_n\}= \{1,\ldots,K\},
$$
but its cardinal function $f$ satisfies $\rng(f)=\{1,2,3\}$.

Let us fix $K\in\mathbb{N}$ and let $x$ be any nondyadic number between $\frac{1}{2^K}$ and $\frac{1}{2^{K-1}}$. Define 
\begin{align*}
x_{n}:=\left\{\begin{array}{ll}
\frac{1}{2^n} & n\leq K-1, \\
x, & n=K, \\
\frac{1}{2^{n-1}} & n\geq K+1 .
\end{array}\right.
\end{align*}
By Theorem \ref{niedwojkowowymierna} we get $\rng(f)=\{1,2,3\}$. On the other hand, $\{ n\in\mathbb{N}\ |\ r_n>x_n\}=\{1,\ldots,K\}$. 
\end{example}

In Example \ref{zawszerazdwatrzy} we observed that if $(x_{n})$ is an interval-filling sequence such that the set $\{n\in\mathbb{N} \mid r_n>x_n\}$ is non-empty, then $\max\rng (f)\geq 3$. Let us now study the largest possible value of $\max\rng (f)$ over all such sequences. It turns out that this largest value is strictly connected with the cardinality of the set $\{n\in\mathbb{N} \mid r_n>x_n\}$.

\begin{theorem}\label{majoranta}
Let $(x_n)$ be an interval-filling sequence and let $f$ be its cardinal function. If 
\begin{align*}
\# \{n\in\mathbb{N} \mid r_n>x_n\} \leq k,
\end{align*}
then $\max \rng (f)\leq 2^{k+1}$
\end{theorem}
\begin{proof}
We use induction on $k$. The case of $k=1$ is covered in Theorem \ref{jednaostra}. Assume that the assertion holds for all $k\leq m$ for some $m$. We will show that it also holds for $m+1$.

Let $\{n\in\mathbb{N} \mid r_n>x_n\}=\{n_1< n_2<\ldots<n_{m+1}\}$. By $f_{k}$ denote the cardinal function of the sequence $(x_{n})_{n>k}$, and let $M_k:=\rng(f_{k})$.
By the inductive assumption, $\max M_{n_1+1}<2^{m+1}$. Using the same idea as in the proof of Theorem \ref{dwojkowowymierna} one can show that $f_{n_1}(x)=f_{n_1+1}(x)+f_{n_1+1}(x-x_{n_{1}})$ for all $x$. Therefore, 
$$
\max M_{n_1}\leq \max\big(M_{n_{1}+1}+M_{n_{1}+1}\big) \leq \max M_{n_{1}+1}+\max M_{n_{1}+1}\leq 2^{m+1}+2^{m+1}=2^{m+2}.
$$  
Clearly, $M_1=M_2=\ldots=M_{n_1}$. This finishes the proof.
\end{proof}

Theorem \ref{majoranta} is the best possible in the sense, that for every $k$ there exists a sequence such that the largest value of range of its cardinal function is equal to $2^{k+1}$. We present it in the next example.

\begin{example}
Let $k\in\mathbb{N}$ and $(x_n^{(k)})=(3^{k-1},3^{k-1},3^{k-2},3^{k-2},3^{k-3},3^{k-3},\ldots,9,9,3,3,1,1,\frac{1}{2},\frac{1}{4},\frac{1}{8},\frac{1}{16},\ldots)$.
\\Let $d$ be any dyadic number in $(0,1)$. Put $x=d+\sum_{m=0}^{k-1} 3^m$. Each summand of $\sum_{m=0}^{k-1} 3^m$ appears twice in the sequence $(x_n^{(k)})$, so  $f(\sum_{m=0}^{k-1} 3^m)= 2^k$. Since $f(d)=2$, we get $f(x)=2^{k+1}$.
Finally note that the set $\{n \mid r_n^{(k)}>x_n^{(k)}\}$ is equal to $\{1,3,5,7,\ldots,2k-3,2k-1\}$, so it has $k$ elements.
\end{example}

To finish this section, let us provide a criterion implying $\max\rng (f)\geq 4$. This improves, under additional assumptions, the bound $\max\rng (f)\geq 3$ from Lemma \ref{conajmniejtrzykrotny}.

\begin{theorem}\label{dwieostre}
Let $(x_n)$ be an interval-filling sequence such that for two  indices $k$ and $m$ such that $k<m$ the following two strict inequalities hold:
\begin{align*}
x_{m} < r_{m} \ \ \ \ \textrm{ and } \ \ \ \ x_k+x_m<r_k.
\end{align*}
Then $\max \rng(f)\geq 4$. 
\end{theorem}
\begin{proof}
For every $k$ let us denote by $f_{k}$ the cardinal function of $(x_{n})_{n>k}$.

Let $k$ and $m$ be the indices from the statement. It follows, that the set $(x_m,r_m)\cap (x_m,r_k-x_k)$ is non-empty. We can repeat the reasoning from the beginning of the proof of Lemma \ref{conajmniejtrzykrotny} to find $x\in (x_m,r_m)\cap (x_m,r_k-x_k)$ such that $f_{m}(x)\geq 2$. This means that there exist two different sets $A,B\subseteq\{m+1,m+2,\ldots\}$ of indices such that $x=\sum_{n\in A}x_n=\sum_{n\in B}x_n$. Let us put
\begin{align*}
y:=x_k+x=\sum_{n\in \{k\}\cup A}x_n=\sum_{n\in \{k\}\cup B}x_n.
\end{align*}

Since $x-x_m\in(0,r_m-x_m)\subseteq (0,r_m)=\mathcal{A}((x_{n})_{n>m})$ by Corollary \ref{CorIntFillA}, one can find a set $C\subseteq\{m+1,m+2,\ldots\}$ such that $x-x_m=\sum_{n\in C}x_n$. Hence,
\begin{align*}
y=x_k+x=x_k+x_m+(x-x_m)=\sum_{n\in \{k,m\}\cup C}x_n.
\end{align*}

In order to get the fourth representation of $y$ we use the assumption $x\in (x_m,r_k-x_k)$, which implies that $y=x_k+x\in (x_k+x_m,r_k)$. We have $y<r_k$, so there exists $D\subseteq\{k+1,k+2,\ldots\}$ such that $y=\sum_{n\in D}x_n$.

Let us show that the four representations defined above are distinct. Indeed the first and the second ones are different by the assumption. The third one is the only one containing both indices, $k$ and $m$. The last representation is the only one without the index $k$. Therefore $f(y)\geq 4$ as desired.
\end{proof}

Note that the sequence $(x_{n})$ constructed in Example \ref{zawszerazdwatrzy} satisfies for all $n\leq K-1$ the relation $r_n-x_n=x=x_K$. Furthermore, the largest index for which the remainder is greater than preceding term is $n=K$, that is, $r_K=\frac{1}{2^{K-1}}>x=x_K$.  Hence, the assumption of Theorem \ref{dwieostre} is not satisfied.

\section{Cardinality invariant sets}

In this section we introduce a new tool which allows us to study cardinal functions from a different point of view. To make the discussion precise, let us introduce do following operation of "adding" a number to a sequence.
\begin{definition}
Let $\textbf{x}=(x_n)$ be a sequence and $y\in\mathbb{R}$. We define $y\uplus \textbf{x}$ to be the sequence $(y_n)$ satisfying $y_1=y$ and $y_{n+1}=x_n$ for all $n\in\mathbb{N}$.
\end{definition}

Now, let us present a motivation. Denote $\textbf{b}:=(1/2^{n})_{n=1}^{\infty}$. In Lemma \ref{lematdodwojkowych} and Theorems \ref{dwojkowowymierna} and \ref{niedwojkowowymierna} we have seen that for every $y\in (0,1)$ the range of the cardinal function of $\mathbf{b}$ is different from the range of the cardinal function of $y\uplus\mathbf{b}$. On the other hand, it is easy to see that the ranges of the cardinal functions of sequences $\mathbf{b}$ and $y\uplus\mathbf{b}$ are the same if $y\geq 1$. A natural question arises, what happens if we replace the set $\mathbf{b}$ by different sequences $\mathbf{x}$, and  what the set of "allowed" real numbers $y$ that do not change the cardinal function can tell us about the sequence $\mathbf{x}$? This suggests the following definition.
\begin{definition}
Let $\textbf{x}=(x_n)$ be a sequence and $f$ be its cardinal function. By the {\it cardinality invariant set} for the sequence $\textbf{x}$ we mean
\begin{align*}
\Cis (\textbf{x}) := \big\{\ y>0\ \big| \textrm{ the cardinal function } g_{y} \textrm{ of } y\uplus \textbf{x} \textrm{ satisfies } \rng(g_{y})=\rng(f)\ \big\}.
\end{align*}
\end{definition}
It appears that the shape of the set $\Cis (\textbf{x})$ is connected with the boundedness of the cardinal function of $\mathbf{x}$.

From the definition, we immediately get the next lemma.

\begin{lemma}\label{trywialnefixy}
Let $\mathbf{x}=(x_n)$ be a 
sequence and $f$ be its cardinal function. Then 
\begin{enumerate}
\item $\left(\sum_{n=1}^{\infty}x_n,\infty\right)\subseteq \Cis (\mathbf{x})$,
\item $2\in\rng (f)$ if and only if $\sum_{n=1}^{\infty}x_n\in \Cis (\mathbf{x})$.
\end{enumerate}
\end{lemma}

The previous discussion implies that $\Cis(\mathbf{b})=[1,\infty)$. This shows that there exists a sequence for which the equality $\Cis (\textbf{x})=[\sum_{n=1}^{\infty}x_n,\infty)$ holds.

Lemma \ref{trywialnefixy} gives a good description of large numbers belonging to $\Cis (\textbf{x})$. It is more interesting to describe small values of in $\Cis (\textbf{x})$ in more details. At first, let us present examples showing that the cardinality invariant set can be larger than $\left[\sum_{n=1}^{\infty}x_n,\infty\right)$. That is, it is possible that $\min\Cis(\textbf{x}) < \sum_{n=1}^{\infty}x_n$ or even more interestingly, that $\Cis(\textbf{x}) = (0,\infty)\setminus\{\sum_{n=1}^{\infty}x_n\}$ (Example \ref{ExampleFixMax2}) or $\Cis(\textbf{x}) = (0,\infty)$ (Example \ref{ExampleFixMax3}). 

\begin{example}
Let $\textbf{x}=(x_n)$ be a sequence such that $x_1$ is a non-dyadic number and $x_{n+1}=\frac{1}{2^n}$ for all $n\in\mathbb{N}$. Clearly $\sum_{n=1}^{\infty}x_n=x_1+1\in \Cis (\textbf{x})$. Now we show that also $1\in \Cis (\textbf{x})$. At first, note that the range of cardinal function is fixed up to homogenity of terms. In particular, the range of the cardinal function of a sequence $(y_n)$ is the same as of $(\frac{1}{2}y_n)$. Moreover, the range of the cardinal function does not depend on the order of terms in a sequence. Therefore, the range of the cardinal function of $1\uplus \textbf{x} = (1,x_{1},x_{2},\ldots )$ is the same as the range of the function of $(\frac{x_{1}}{2},\frac{1}{2},\frac{1}{4},\ldots )$, which is equal to $\{1,2,3\}$ by Theorem \ref{niedwojkowowymierna} (since $x_{1}$ and hence $\frac{x_{1}}{2}$ is assumed to be non-dyadic).
\end{example}


\begin{example}\label{ExampleFixMax2}
Let $\mathbf{x}=(x_{n})$ be defined as $x_{2n-1}=x_{2n}=\frac{1}{2^n}$ for all $n\in\mathbb{N}$, and let $f$ be the cardinal function of $\mathbf{x}$. It was shown in \cite{cardfun} that $\mathcal{A}(\mathbf{x})=[0,2]$, $f(0)=f(2)=1$ and $f(x)=\mathfrak{c}$ for all $x\in (0,2)$. Thus $\rng(f)=\{1,\mathfrak{c}\}$. By Lemma \ref{trywialnefixy} we know that $(2,\infty)\subseteq \Cis (\textbf{x})$ and $2\notin \Cis (\textbf{x})$. Let $y\in (0,2)$ and consider $\mathbf{y}:= y\uplus\mathbf{x}$. By $g$ denote the cardinal function of $\mathbf{y}=(y_n)$. Then we have $A(\mathbf{y})=[0,2+y]$ and the following equalities hold:
\begin{displaymath}
g(x) = f(x)+f(x-y)= \left\{ \begin{array}{ll} 
1+0=1, & \textrm{if $x=0$},\\ 
\mathfrak{c}+0=\mathfrak{c}, & \textrm{if $x\in(0,y)$}, \\  
\mathfrak{c}+1= \mathfrak{c}, & \textrm{if $x=y$},\\  
\mathfrak{c} +\mathfrak{c}= \mathfrak{c}, & \textrm{if $x\in (y,2)$},\\ 
1+ \mathfrak{c}= \mathfrak{c}, & \textrm{if $x=2$},\\ 
0+ \mathfrak{c}= \mathfrak{c}, & \textrm{if $x\in (2,2+y)$},\\ 
0+1=1, & \textrm{if $x=2+y$}.
\end{array} \right.
\end{displaymath}
Thus $\rng(g)=\{1,\mathfrak{c}\}=\rng(f)$. Hence, we obtain $\Cis (\textbf{x})=(0,2)\cup(2,\infty)$, which shows that there exists a sequence for which the equality $\Cis (\textbf{x})=(0,\infty)\setminus \{\sum_{n=1}^{\infty}x_n\}$  holds.
\end{example}

\begin{example}\label{ExampleFixMax3}
Now we construct a sequence $\mathbf{y}=(y_n)$ for which $\Cis (\textbf{y})=(0,\infty)$. The example comes from \cite{cardfun}. Let $(x_n)$ satisfies the following property \eqref{PropertyA}: 
\begin{equation}\label{PropertyA}
x_n>\sum_{i=n+1}^{\infty} (x_i+r_{i-1}) = \sum_{i=n+1}^{\infty} (i-n+1)x_{i} \tag{A}
\end{equation}
for all $n\in\mathbb{N}$. Let further a sequence $(y_{n})$ be define by the equalities:
\begin{align*}
y_{2n-1} = r_{n-1} \ \ \ \ \textrm{ and } \ \ \ \ y_{2n} = x_{n}
\end{align*}
for each $n\in\mathbb{N}$, and $f$ be the cardinal function of $(y_n)$. By \cite[Theorem 2.2]{cardfun} we know that $\rng(f)=\mathbb{N}\cup\{\omega,\mathfrak{c}\}$, that is, the largest possible. From the definition, it is easy to see that $(x_n)_{n\geq k}$ satisfies the property \eqref{PropertyA} for any $k\in\mathbb{N}$. Hence, for all $k\in\mathbb{N}$ we get that $\rng(f_{2k})=\mathbb{N}\cup\{\omega,\mathfrak{c}\}$, where $f_{2k}$ is the cardinal function of $(y_n)_{n\geq 2k+1}$. 

Clearly $\Cis (\textbf{y})\supseteq [\sum_{n=1}^{\infty}y_n,\infty)$. Let $y_0\in(0,\sum_{n=1}^{\infty}{y_n})$ and $g$ be the cardinal function of $(y_n)_{n\geq 0}$. Then there exists $k\in\mathbb{N}$ such that 
\begin{align*}
y_0>\sum_{n=2k+1}^{\infty}y_{n}
\end{align*}
(note that the convergence of $\sum_{n=1}^{\infty}y_{n}$ is guaranteed by the property \eqref{PropertyA} of $(x_n)$). Thus $g\vert_{\A((y_n)_{n\geq 2k+1}})$ is equal to $f_{2k}$. Hence $\rng(g)\supseteq \rng(f_{2k})$, so
\begin{align*}
\rng(g)=\mathbb{N}\cup\{\omega,\mathfrak{c}\}=\rng(f),
\end{align*}
which means that $y_0\in \Cis (\textbf{y})$. Therefore, $\Cis (\textbf{y})=(0,\infty)$.
\end{example}

Now we prove some results relating the set $\Cis (\mathbf{x})$ with some properties of the cardinal function of $\mathbf{x}$.

\begin{proposition}\label{zawieranaturalne}
Let $\mathbf{x}=(x_{n})$ be a sequence and $f$ be its cardinal function. If $\Cis (\mathbf{x})=(0,\infty)$, then $\rng (f)\supseteq\mathbb{N}$. 
\end{proposition}
\begin{proof}
Clearly $1\in \rng (f)$. We prove that if $k\in \rng (f)$ for some $k\in\mathbb{N}$, then $k+1\in \rng (f)$. Let $y\in \mathcal{A}(x_n)$ be such that $f(y)=k$. Consider the sequence $\mathbf{y}:=y\uplus \mathbf{x}$ and let $g$ be its cardinal function. Then easily $g(y)=k+1$, and thus $k+1\in \rng(g)$. But $y\in \Cis (\textbf{x})$, so $k+1\in \rng(f)$. 
\end{proof}

\begin{corollary} 
Let $\mathbf{x}$ be a sequence with bounded cardinal function. Then $\Cis (\mathbf{x})\subsetneq (0,\infty)$.
\end{corollary}

\begin{proposition}
Let $\mathbf{x}$ be a sequence such that its cardinal function $f$ is bounded. Then for every $t\in \Cis (\mathbf{x})\cap\mathcal{A}(\mathbf{x})$ we have $f(t)<\max\rng (f)$.
\end{proposition}
\begin{proof}
Suppose that there exists $y\in \Cis (\mathbf{x})$ such that $f(y)=\max\rng (f)$. Let $g$ be the cardinal function of the sequence $y\uplus \textbf{x}$. Then
\begin{align*}
g(y)=f(y)+f(0) = \max\rng (f) + 1 >\max\rng (f).
\end{align*}
Therefore, $\rng (f)\neq\rng (g)$ contradicting the assumption $y\in\Cis(\mathbf{x})$.
\end{proof}



Let $\textbf{x}$ be a sequence with cardinal function $f$. In Proposition \ref{zawieranaturalne} we  showed in particular that if $\Cis (\textbf{x})=(0,\infty)$, then $\rng (f)$ is unbounded. 
Let us recall some open and still unsolved problems connected with the notion of unboundedness of $\rng (f)$ posed in \cite{cardfun}.
At first, note that $\rng (f)$ is unbounded if and only if it contains infinitely many natural numbers or if it contains $\omega$ or $\mathfrak{c}$. We still do not know if the infiniteness of $\rng (f)$ implies that at least one of $\omega$ or $\mathfrak{c}$ belongs to $\rng (f)$ \cite[Problem 2.12]{cardfun}. The questions if $\omega\in \rng (f)$ implies that $\rng (f)$ is infinite or that $\mathfrak{c}\in \rng (f)$ remain without answers \cite[Problems 2.9 and 2.13]{cardfun}. On the other hand, there exist examples for which $M=\{1,\mathfrak{c}\}$ or $M=\mathbb{N}\cup\{\mathfrak{c}\}$ (for example in \cite[Proposition 2.5]{cardfun}).

\section{Intervals and finite sums of closed intervals}

In this section, we will continue using the following notation. If $f$ is the cardinal function of $(x_n)_{n\geq 0}$, then $f_{k}$ is the cardinal function of $(x_n)_{n\geq k}$.

The results in Theorem \ref{musibycdwojka} can be restated as follows:
\\\emph{Let $M\in\mathcal{I}_1$. Then at least one of the conditions hold: $2\in M$ or $M$ is unbounded.}  
\\Now we prove the same property for $\mathcal{I}$. 

\begin{theorem}\label{ThmMusibycdwojka2}
Let $M\in\mathcal{I}$. Then at least one of the following hold:
\begin{itemize}
\item $2\in M$;
\item $M$ is unbounded.
\end{itemize}
\end{theorem}
\begin{proof}
Note that the claim has already been proved in the case of a single interval in Theorem \ref{musibycdwojka}.

Let $\mathbf{x}=(x_n)$ be any sequence with $\mathcal{A}(\mathbf{x})$ being a finite union of at least two closed intervals. 
From Theorem \ref{ThmKakeya} we know that this is equivalent to the condition that $x_n\leq r_n$ holds for all but finitely many $n$.

Let $k$ be the minimal index for which $x_n\leq r_n$ holds for every $n>k$. Let us denote 
\begin{align*}
F_k=\left\{\ \sum_{n=1}^{k}\varepsilon_n x_n\  \Bigg|\ (\varepsilon_n)\in\{0,1\}^k\ \right\}=\{\sigma_0<\sigma_1<\ldots<\sigma_p\}
\end{align*}
and
\begin{align*}
\mathcal{A}_{k+1}(\mathbf{x})=\left\{\ \sum_{n=k+1}^{\infty}\varepsilon_n x_n\ \ \Bigg|\ (\varepsilon_n)\in\{0,1\}^{\mathbb{N}}\ \right\}=\bigg[0,\sum_{n=k+1}^{\infty} x_n \bigg]
\end{align*}
Then 
\begin{align*}
\mathcal{A}(\mathbf{x})=F_k + \mathcal{A}_{k+1}(\mathbf{x})=\bigcup_{i=0}^{p} (\sigma_i + \mathcal{A}_{k+1}(\mathbf{x}))= \mathcal{A}_{k+1}(\mathbf{x})\cup\bigcup_{i=1}^{p} (\sigma_i + \mathcal{A}_{k+1}(\mathbf{x})).
\end{align*}

Note that our choice of $k$ implies
\begin{align*}
\sigma_1 = x_k > r_{k} = \sum_{n=k+1}^{\infty}x_n,
\end{align*}
so the sets  $\mathcal{A}_{k+1}(\mathbf{x})$ and $\bigcup_{i=1}^{p} (\sigma_i + \mathcal{A}_{k+1}(\mathbf{x}))$ are disjoint.

Now observe that the achievement set of the sequence $(x_{n})_{n\geq k+1}$ is precisely the (single) interval $\mathcal{A}_{k+1}(\mathbf{x})$. Hence, we can apply Theorem \ref{musibycdwojka} to it and get that either $2\in \rng (f_{k+1})$ or $\rng (f_{k+1})$ is unbounded.

Observe, that $\mathcal{A}_{k+1}(\mathbf{x})\subseteq \mathcal{A}(\mathbf{x})$ and $f\mid_{\mathcal{A}_{k+1}(\mathbf{x})} = f_{k+1}$. Hence, $\rng (f_{k+1})\subseteq \rng (f)$ and the statement follows.


 
\end{proof}

Theorem \ref{ThmMusibycdwojka2} implies several interesting facts. We present some of them in the next corollary.

\begin{corollary}\label{CorManyCors}
\begin{enumerate}[label={(\arabic*)},ref={\thecorollary(\arabic*)}]
	\item For any $m\geq 3$ we have $\{1,m\}\notin \mathcal{I}$, \label{nienalezydosumy}
	\item The following inclusion holds: 
	\begin{align*}
	\mathcal{I}\subseteq \{M \ | \ \text{there exists} \ M_0\in \mathcal{I}_1 \ \text{such that} \ M\supseteq M_0 \}.
	\end{align*}		
	It other words, the ranges of cardinal functions being finite unions of closed intervals can only be found among the supersets of the ranges in the single interval case. \label{nadzbiory}
	\item Let $\mathbf{x}$ be a sequence such that $\mathcal{A}(\mathbf{x})$ is a finite union of closed intervals and let $M$ be the range of its cardinal function. If $M$ is bounded, then $2\in M$. \label{dwojkawsumachprzedzialow} 
\end{enumerate}
\end{corollary}
\begin{proof}
	\begin{enumerate}
		\item Immediate consequence of Theorem \ref{ThmMusibycdwojka2}.
		\item This follows from the reasoning presented in the proof of Theorem \ref{ThmMusibycdwojka2}.
		\item This is a rephrased statement of Theorem \ref{ThmMusibycdwojka2}.
	\end{enumerate}
\end{proof}

Corollary \ref{dwojkawsumachprzedzialow} allows us to exclude from $\mathcal{I}$ the same sets as Theorem \ref{musibycdwojka} has done for $\mathcal{I}_1$. Some of them are mentioned in Corollary \ref{jedentrzyczteryprzedzial}.

\bigskip

Let us focus for a while on studying relations between the sets $\mathcal{I}$ and $\mathcal{I}_{1}$. Obviously, $\mathcal{I}_{1}\subseteq \mathcal{I}$. Lemma \ref{dolozskonczony} allows us to strengthen this inclusion. As a consequence, we show that some sets belong to $\mathcal{I}$.

\begin{corollary}\label{CorFI1inI}
	\begin{enumerate}[label={(\arabic*)},ref={\thecorollary(\arabic*)}]
	\item The following inclusion holds: 
	\begin{align*}
	\mathcal{I}\supseteq \mathcal{F}\cdot \mathcal{I}_1=\{A\cdot B\ |\ A\in \mathcal{F}, B\in \mathcal{I}_1\}.
	\end{align*}
	In particular, 
	$\{1,2,4\},\{1,2,3,6\}, \{1,2,3,4,6\}\in \mathcal{I}$. \label{sumaprzedzialowskonczonyprzedzial}
	\item If $M\in \mathcal{F}$, then $M\cup\{\mathfrak{c}\}\in \mathcal{I}$.
	\end{enumerate}
\end{corollary}
\begin{proof}
	\begin{enumerate}
	\item By Lemma \ref{dolozskonczony} we quickly get
	\begin{align*}
		\mathcal{F}\cdot\mathcal{I}_{1} \subseteq \mathcal{F}\cdot\mathcal{I} = \mathcal{I}.
	\end{align*}
	
	The "In particular" part follows from the observation that $\{1,2\}\in \mathcal{I}_{1}$, $\{1,2\},\{1,3\},\{1,2,3\}\in\mathcal{F}$, and
		\begin{align*}
			\{1,2\}\cdot\{1,2\} = \{1,2,4\}, \ \ \ \{1,2\}\cdot \{1,3\} = \{1,2,3,6\}, \ \ \ \{1,2\}\cdot \{1,2,3\} = \{1,2,3,4,6\}.
		\end{align*}
		Indeed, $\{1,2\}\in\mathcal{I}_{1}$ by Lemma \ref{lematdodwojkowych}, and it is easy to verify that sets $\{1,2\}$, $\{1,3\}$ and $\{1,2,3\}$ are the ranges of the cardinal functions of (finite) sequences $(1,1)$, $(1,1,1)$ and $(1,1,2,3)$, respectively.
		\item Follows from the fact that $\{1,\mathfrak{c}\}\in \mathcal{I}_1$ and Corollary \ref{sumaprzedzialowskonczonyprzedzial}.
	\end{enumerate}
\end{proof}

Based on Corollary \ref{dolozskonczony}, it is natural to ask whether $\mathcal{I}=\mathcal{F}\cdot\mathcal{I}_{1}$? We were unable to find an answer. Therefore, we leave it as a problem for possible future research.

\begin{problem}\label{ProblemI=FI_1}
    Does the equality $\mathcal{I}=\mathcal{F}\cdot \mathcal{I}_1$ hold?
\end{problem}

The affirmative answer to Problem \ref{ProblemI=FI_1} may have an interesting consequence. In order to phrase it, let us introduce the following notion.

\begin{definition}
A set $P$ is called a {\it prime set} if the equality $P=A\cdot B$ implies that $A=\{1\}$ or $B=\{1\}$.
\end{definition}

\begin{example}
All of the sets $\{1\}, \{1,2\}, \{1,3\}, \{1,2,3\}, \{1,3,4\},\{1,2,3,4\}$ and $\{1,m\}$ for each $m\geq 4$ are prime sets. Set $\{1,2,4\}$ is not a prime set because  $\{1,2,4\}= \{1,2\}\cdot  \{1,2\}$.
\end{example}

\begin{proposition}
Let $P$ be a prime set. If $P\in \mathcal{F}\cdot \mathcal{I}_{1}$ then $P\in\mathcal{I}_{1}$.
\end{proposition}
\begin{proof}
Since $P\in \mathcal{F}\cdot \mathcal{I}_{1}$, there exist sets $A\in\mathcal{F}$ and $B\in\mathcal{I}_{1}$ such that $P=A\cdot B$. By the assumption $P$ is prime, so $A=\{1\}$ or $B=\{1\}$. The latter is impossible since $\{1\}\notin \mathcal{I}_{1}$. Therefore, $A=\{1\}$ and $P=B\in\mathcal{I}_{1}$.
\end{proof}

\begin{corollary}\label{CorBasedOnProblem}
Assume that the answer to Problem \ref{ProblemI=FI_1} is affirmative and let $P$ be a prime set. If $P\in\mathcal{I}$ then $P\in\mathcal{I}_{1}$.
\end{corollary}

Corollary \ref{CorBasedOnProblem} may suggest a strong connection (maybe even the equality) between $\mathcal{I}_{1}$ and $\mathcal{I}$. Therefore, now we focus on showing some differences between these two sets. In order to distinguish $\mathcal{I}_1$ and $\mathcal{I}$, let us define the set $\mathcal{I}_{>1}$ of all ranges of cardinal functions for achievement sets, which are the finite union of at least two closed intervals. Clearly $\mathcal{I}=\mathcal{I}_{>1}\cup \mathcal{I}_1$. Now we give some properties of $\mathcal{I}_{>1}$.

\begin{lemma}
The inclusion $\mathcal{I}_1\subseteq \mathcal{I}_{>1}$ holds.
\end{lemma}
\begin{proof}
Let $\mathbf{x}=(x_{n})$ be an interval-filling sequence with cardinal function $f$. Let $c$ be any number greater than $\sum_{n=1}^{\infty}x_{n}$. Let $\mathbf{y}:=c\uplus \mathbf{x}$ and $g$ be the cardinal function of $\mathbf{y}$. Then $A(\mathbf{y})=A(\mathbf{x})\cup (c+A(\mathbf{x}))$ and both sets in the sum are disjoint. Thus $A(\mathbf{y})$ is a union of two closed intervals. Moreover, for every $t\in A(\mathbf{x})$ we have $g(t)=f(t)=g(c+t)$. Hence, $\rng(g)=\rng(f)$, so $\rng (f)\in\mathcal{I}_{>1}$.
\end{proof}

\begin{corollary}
The equality $\mathcal{I}_{>1}=\mathcal{I}$ holds.
\end{corollary}







Now we show that $\mathcal{I}\neq \mathcal{I}_1$, that is, there exists a range obtainable for a finite union of closed intervals which cannot be achieved for any single interval.

\begin{theorem}\label{sumaprzedzialowalenieprzedzial}
We have $\{1,2,4\}\in \mathcal{I}\setminus \mathcal{I}_1$.
\end{theorem}
\begin{proof}
By Corollary \ref{sumaprzedzialowskonczonyprzedzial} we know that $\{1,2,4\}\in \mathcal{I}$, so it is enough to show that $\{1,2,4\}\notin \mathcal{I}_1$.

Suppose that there exists an interval-filling sequence $(x_n)$ such that its cardinal function $f$ satisfies $\rng (f)=\{1,2,4\}$. Recall that we assume that $(x_{n})$ is nonincreasing. By Corollary \ref{geometrycznyogon} there exist a positive integer $k$ and a nonzero $c$ such that $x_n=\frac{c}{2^n}$ for $n\geq k$. Lemma \ref{LemCritIntervFill} yields
\begin{align*}
x_{k-1}\leq r_{k-1}=\sum_{n=k}^{\infty} x_{n} = \frac{c}{2^{k}} \sum_{n=0}^{\infty} \frac{1}{2^{n}} = \frac{c}{2^{k-1}} =2x_{k}.
\end{align*}
By Theorems \ref{dwojkowowymierna} and \ref{niedwojkowowymierna} it is clear that $k\geq 3$.

Let us consider two cases depending on whether $x_{k-1}\in \frac{c}{2^{k-1}}\mathcal{D}$ or not (recall that $\mathcal{D}$ denotes the set of all dyadic numbers in the unit interval).
\begin{enumerate}
	\item If $x_{k-1}\in c\mathcal{D}$ then Theorem \ref{dwojkowowymierna} yields $f_{k-1}(x_{k-1})=3$. Note that the value $3$ does not belong to $\rng (f)$, so we need to find another representation of $x_{k-1}$. It is impossible if $x_{k-2}>x_{k-1}$. Therefore $x_{k-2}=x_{k-1}$. However, then for every $d\in \frac{c}{2^{k-1}}\mathcal{D}$ such that $x_{k-1}+d\in \frac{c}{2^{k-1}}\mathcal{D}\cap(0,2x_k)$ we have 
\begin{align*}
f(x_{k-1}+d)\geq f_{k-2}(x_{k-1}+d)= f_{k}(x_{k-1}+d)+2f_k(d)+f_{k}(d-x_{k-1})\geq 2+2\cdot 2+0 =6,
\end{align*}
where equalities $f_{k}(x_{k-1}+d)=f(d)=2$ follow from the formula for $g(x)$ obtained in the proof of Theorem \ref{dwojkowowymierna}. This leads to a contradiction.
	\item If $x_{k-1}\notin c\mathcal{D}$ then Theorem \ref{niedwojkowowymierna} implies that for every $\varepsilon >0$ there is an element $y\in \left(\frac{c}{2^{k-1}}D\right)\cap(x_{k-1},x_{k-1}+\varepsilon)$ such that $f_{k-1}(y)=3$. Hence, it is again impossible that $x_{k-2}>x_{k-1}$, so $x_{k-2}=x_{k-1}$. Then for any $d\in \frac{c}{2^{k-1}}D$ such that $x_{k-1}+d<2x_k$ we obtain that 
\begin{align*}
f(x_{k-1}+d)\geq f_{k-2}(x_{k-1}+d)\geq f_{k}(x_{k-1}+d)+ 2f_k(d)=1+2\cdot 2=5
\end{align*}
by the expression for $g(x)$ obtained in the proof of Theorem \ref{niedwojkowowymierna}, which is a contradiction.
\end{enumerate}
In both considered cases we obtained a contradiction. Therefore, an interval-filling sequence $(x_n)$ with its cardinal function $f$ satisfying $\rng (f)=\{1,2,4\}$ cannot exist.
\end{proof}


The idea from the proof of Theorem \ref{sumaprzedzialowalenieprzedzial} can be pushed further.

\begin{corollary}\label{niematrojkitojestszostka}
If $M\in \mathcal{I}_1$ is bounded and $3,4,5,\ldots,m\notin M$ then $\max M\geq 2m$.   
\begin{proof}
We use the same notation as in the proof of Theorem \ref{sumaprzedzialowalenieprzedzial} and consider the analogous two cases:
\begin{enumerate}
	\item Assume that $x_{k-1}\in \frac{c}{2^{k-1}}\mathcal{D}$ and observe that it leads to $x_{k-m+1}=\ldots=x_{k-2}=x_{k-1}$. Indeed, we have $f_{k-1}(x_{k-1})=3$, and therefore need to have $x_{k-2}=x_{k-1}$. But then $f_{k-2}(x_{k-1})=4$, so again $x_{k-1}=x_{k-2}$. By repeating this reasoning we get the claimed equalities. Therefore, if $d \in\frac{c}{2^{k-1}}\mathcal{D}$ is such that $x_{k-1}+d<2x_{k}$, we obtain
	\begin{align*}
		f_{k-m+1}(x_{k-1}+d)\geq f_k(x_{k-1}+d) + (m-1)f_k(d)= 1 +(m-1)\cdot 2=2m.
	\end{align*}
	\item The case of $x_{k-1}\notin c\mathcal{D}$ is analogous to the previous one. We get $x_{k-m}=\ldots=x_{k-2}=x_{k-1}$ and then
	\begin{align*}
	f_{k-m}(x_{k-1}+d)\geq  f_{k}(x_{k-1}+d) + m\cdot f_k(d)=2m+1.
	\end{align*}
\end{enumerate}
The result follows.
\end{proof}
\end{corollary}

\begin{corollary}
The following sets do not belong to $\mathcal{I}_{1}$:
$\{1,2,5\}$, $\{1,2,6\}$, $\{1,2,4,5\}$, $\{1,2,5,6\}$, $\{1,2,5,7\}$, $\{1,2,5,6,7\}$.
\end{corollary}

\section{More about interval-filling sequences}

In this section we show several examples of ranges $M$ from $\mathcal{I}_1$ with $\max M\leq 6$. In the proof of the following lemma (which is interesting on its own and will be used later) we present a technique, which will be crucial in our constructions.

\begin{lemma}\label{mniejnizszesc}
Let $\mathbf{x}=(x_n)$ be a sequence with cardinal function $f$. Let $z$ and $w$ be positive real numbers. Assume that one of the following holds:
\begin{enumerate}
\item $z+w>\sum_{n=1}^{\infty}x_n$,
\item $z=w= \frac{1}{2}\sum_{n=1}^{\infty}x_n$.
\end{enumerate}
Let $h$ be the cardinal function of the sequence $z\uplus (w\uplus \textbf{x}) = (z,w,x_{1},x_{2},\ldots )$. Then 
\begin{align*}
\max (h)\leq 3\max (f).
\end{align*}
In particular, if $\rng (f)=\{1,2\}$ then $\max (h)\leq 6$.
\end{lemma} 
\begin{proof}
Note that $h(x)=f(x)+f(x-z)+f(x-w)+f(x-z-w)$. If $z+w>\sum_{n=1}^{\infty}x_n$ then at most one of the elements $x$ and $x-z-w$ belongs to $\mathcal{A}(x_n)$. Hence  
\begin{align*}
h(x)=\big(f(x)+f(x-z-w)\big)+f(x-z)+f(x-w)\leq \max(f)+\max(f)+\max(f)=3\cdot\max(f).
\end{align*}
This finishes the proof if $z+w>\sum_{n=1}^{\infty}x_{n}$.

Now assume that $z=w=\frac{1}{2}\sum_{n=1}^{\infty}x_n$. Then for $x\neq \sum_{n=1}^{\infty}x_n$ we proceed in the same way, as in the first case, and obtain $f(x)\leq 3$. 

It remains to prove that $f\left(\sum_{n=1}^{\infty}x_n\right)\leq 3$. Consider two subcases:
\begin{enumerate}
	\item If $\rng (f)\neq\{1\}$, then we have
		\begin{align*}
		h\left(\sum_{n=1}^{\infty}x_n\right)=f\left(\sum_{n=1}^{\infty}x_n\right)+2\cdot f\left(\frac{1}{2}\sum_{n=1}^{\infty}x_n\right)+f(0)\leq 1+2\cdot \max(f) +1\leq 3\cdot \max(f).
		\end{align*}
	\item If $\rng (f)=\{1\}$, at first observe that $f\left(\frac{1}{2}\sum_{n=1}^{\infty}x_n\right)$ is even or is equal to $\omega$ or $\mathfrak{c}$. Therefore, since $\rng f=\{1\}$, we have to have $\frac{1}{2}\sum_{n=1}^{\infty}x_n\notin \mathcal{A}(\mathbf{x})$. Hence, $h\left(\frac{1}{2}\sum_{n=1}^{\infty}x_n\right) = 2\leq 3$.
\end{enumerate}
\end{proof}

\begin{proposition}\label{PropSetsI_1}
$\{1,2,3\}$, $\{1,2,3,4\}$, $\{1,2,3,4,5\}$, $\{1,2,3,4,6\}$, $\{1,2,3,4,5,6\}\in\mathcal{I}_{1}$.
\end{proposition}
\begin{proof}
For every set $M$ from the statement, we will find a sequence $\mathbf{x}$ with the cardinal function $h$ such that $\rng (h)=M$. At first, let $\mathbf{b}= \left(\frac{1}{2^{n}}\right)$ and $f$ be its cardinal function. In the constructions, we will consider sequences of the form $\mathbf{x} = z\uplus (w\uplus \mathbf{b})$ for several choices of $z$ and $w$. Note that $\rng (f)=\{1,2\}$ by Lemma \ref{lematdodwojkowych}, so $\rng (h)\leq 6$ by Lemma \ref{mniejnizszesc}.

Now let us proceed case by case.
\begin{enumerate}
	\item For $\{1,2,3\}$: let $z>w$ be non-dyadic numbers such that $z\in (1,1+w)$, $w\in(\frac{1}{2},1)$, and $z-w$ is non-dyadic. Then clearly $\mathcal{A}(\mathbf{x}) = [0,1+z+w]$ and $1,2\in \rng (h)$. Moreover, observe that for every $x$:
	\begin{align*}
	h(x) = f(x) + f(x-z) + f(x-w) + f(x-z-w),
	\end{align*}		
	and at most one of the numbers $x$, $x-z$, $x-w$, $x-z-w$ is dyadic. Moreover, the relations between sizes of $z$ and $w$ imply that always at most two of these these numbers lie in the interval $[0,1]$. Indeed,
	\begin{itemize}
	\item if $x\in [0,1]$ then $x-z<0$ and $x-z-w<0$,
	\item if $x\in (1,1+w]$ then $x-z-w<0$ and $x>1$,
	\item if $x\in [1+w,1+z+w]$ then $x>1$ and $x-w>1$.
	\end{itemize}	
	 Therefore, Lemma \ref{lematdodwojkowych} yields $h(x)\leq 3$. Observe, that also $h(x)=3$ for every dyadic $x\in (0,1)$. Hence, $\rng (h)=\{1,2,3\}$.
	\item For $\{1,2,3,4\}$: here we use a similar choice, as in the previous construction. The only change is that now we require $z-w$ to be dyadic. Then at most two of the numbers $x$, $x-z$, $x-w$, $x-z-w$ can be dyadic, so $h(x)\leq 4$. Moreover, $h(x)=3$ for every dyadic $x\in (w,1)$ and $h(x)=4$ for example for every $x=z+\frac{1}{2^{m}}$, where $m$ is any number such that $z+\frac{1}{2^{m}} < 1+w$. Hence, $\{1,2,3,4\}\in\mathcal{I}_{1}$.
	\item For $\{1,2,3,4,5\}$: let us take $z=w\in (\frac{1}{2},1)$ be any non-dyadic number. Then $\mathcal{A}(\mathbf{x})=[0,1+2z]$, $1,2\in\rng (h)$, and
	\begin{align*}
	h(x)=f(x)+2\cdot f(x-z)+f(x-2z).
	\end{align*}
	At most on of the values $x$, $x-z$, $x-2z$, so
	\begin{align*}
	h(x)\leq \max\{2+2\cdot 1, 1 +2\cdot 2\}=5.
	\end{align*}
	It remains to show, that $h(x)$ can be equal to $3$, $4$ and $5$.
	\begin{itemize}
		\item $h(x)=3$ for $x=z$.
		\item $h(x)=4$ for $x=1-\frac{1}{2^{m}}$, where $m$ is any positive integer such that $z<1-\frac{1}{2^{m}}$,
		\item $h(x)=5$ for $x=z+\frac{1}{2^{m}}$, where $m$ is such that $z+\frac{1}{2^{m}} < 1$.
	\end{itemize}
	\item For $\{1,2,3,4,6\}$: let $z=w=\frac{1}{2}$. Clearly $\mathcal{A}(\mathbf{x})=[0,2]$. Moreover, for every $x$:
	\begin{align*}
		h(x) = f(x)+2\cdot f\left(x-\frac{1}{2}\right)+f(x-1).
	\end{align*}
	Using the above equality, we can write down a formula for $h(x)$. For $x\in[0,1]$ we have:
	\begin{align*}
	h(x) = \left\{
	\begin{array}{ll}
		1, & \textrm{ if } x=0, \\
		2, & \textrm{ if } x\in (0,\frac{1}{2})\cap \mathcal{D},\\
		3, & \textrm{ if } x\in [0,1]\setminus \mathcal{D},\\
		4, & \textrm{ if } x=\frac{1}{2},\\
		6, & \textrm{ if } x\in (\frac{1}{2},1]\cap\mathcal{D}.
	\end{array}	
	\right.
	\end{align*}
	If $x\in (1,\frac{3}{2})$ then $h(x)=2f\left(x-\frac{1}{2}\right) +f(x-1)$. We have that either both, $x-\frac{1}{2}$ and $x-1$, are dyadic, or non of them. Therefore, in this case $h(x)\in \{3,6\}$. Then $h\left(\frac{3}{2}\right)= 2f(1) + f\left(\frac{1}{2}\right) = 4$. Finally, if $x\in (\frac{3}{2},2]$ then $h(x) = f(x-1)\in\{1,2\}$. Therefore, $\rng (h) =\{1,2,3,4,6\}$.
	\item For $\{1,2,3,4,5,6\}$: we take $z=w\in\ (\frac{1}{2},1)$ to be any dyadic number. Lemma \ref{mniejnizszesc} yields $\max (h)\leq 6$. Moreover, using the same reasoning, as in the proof of the case $\{1,2,3,4,5\}$, we get
	\begin{itemize}
		\item $h(0)=1$,
		\item $h(x)=2$ for every dyadic $x\in (0,z)$,
		\item $h(x)=3$ for every non-dyadic $x\in (z,1)$, 
		\item $h(z)=4$,
		\item $h(1)=5$,
		\item $h(x)=6$ for $x=z+\frac{1}{2^{m}}$, where $m$ is a positive integer such that $z+\frac{1}{2^{m}} < 1$.
	\end{itemize}
	Thus $\rng (h) = \{1,2,3,4,5,6\}$.
\end{enumerate}
\end{proof}

\begin{remark}\label{RemarkIntFilSetsConstructions}
In the cases of $\{1,2,3,4,5\}$ and $\{1,2,3,4,5,6\}$, one can use a more general construction: instead of taking $z=w\in (\frac{1}{2},1)$, we can take $\frac{1}{2}<w<z<1$ and assume that:
\begin{enumerate}
	\item for $\{1,2,3,4,5\}$: $z$ and $w$ are non-dyadic, $z-w$ is dyadic,
	\item for $\{1,2,3,4,5,6\}$: $z$ and $w$ are dyadic.
\end{enumerate}

Similarly, in order to obtain $\{1,2,3,4\}$ one can consider $\frac{1}{2}< w\leq z <1$ such that all $z$, $w$ and $z-w$ are non-dyadic. 
\end{remark}

\bigskip

Now we prove the following improvement of Corollary \ref{niematrojkitojestszostka} in the case of $m=3$.

\begin{theorem}\label{conajmniejsiedem}
Assume that $M\in \mathcal{I}_1$ is bounded, $M\neq \{1,2\}$ and $3\notin M$. Then $\max M\geq 7$.
\begin{proof}
Let $(x_n)$ be an (nonincreasing) interval-filling sequence with cardinal function $f$ with range $M$ satisfying all the given assumptions. 
Since $M$ is bounded, we know from Corollary \ref{geometrycznyogon} that $x_n=\frac{c}{2^n}$ for all $n\geq k$, where $k$ is some positive integer. Multiplying all the terms $x_{n}$ does not change $M$, so we can assume that $c=1$. Assume further that $k$ is the minimal value such that the tail $(x_n)_{n\geq k}$ is geometric.  Note that $M\neq\{1,2\}$, so $k\geq 2$ by Lemma \ref{lematdodwojkowych} and from Theorems \ref{dwojkowowymierna} and \ref{niedwojkowowymierna} we know that also $k\geq 3$. Let us consider two cases depending on whether $x_{k-1}$ is in the set $\mathcal{D}$ or not.

If $x_{k-1}=d\in \mathcal{D}$ (note that $x_{k-1}\leq r_{k-1}=1$ by Lemma \ref{LemCritIntervFill}) then $f_{k-1}(d)=3$. Moreover, if $x_{k-2}>d$ then also $f(d)=3$, which is impossible since we assume that $3\notin M$. Therefore, $x_{k-2}=d$. Hence, $(x_{n})_{n\leq k-2} = d\uplus (d\uplus \mathbf{b})$. By the proof of Proposition \ref{PropSetsI_1}, we know that for every $e\in (d,1)\setminus\mathcal{D}$ we have $f_{k-2}(e)=3$. Hence, if $k=3$ then $3\in M$, a contradiction. Hence, $k\geq 4$. If $x_{k-3}>d$, we can find $e\in (d,1)\setminus \mathcal{D}$ such that $d<e<x_{k-3}$. For such an $e$ we have $f (e)= f_{k-3}(e)=3$, a contradiction. Therefore, $x_{k-3}=d$. 

Let $m$ be such that $d+\frac{1}{2^{m}} < 2x_{k} = \frac{1}{2^{k-1}}$. We get
\begin{align*}
f\left(d+\frac{1}{2^m}\right)\geq f_{k-3}\left(d+\frac{1}{2^m}\right)=f_{k}\left(d+\frac{1}{2^m}\right)+3f_{k}\left(\frac{1}{2^m}\right)=2+3\cdot 2 = 8.
\end{align*}

Let us consider the second case, that is, $x_{k-1}=e\notin \mathcal{D}$. Then for every $x$ we have
\begin{align*}
f_{k-1}(x) = f_{k}(x) + f_{k}(x-e).
\end{align*}
In particular, if $m$ is such that $e+\frac{1}{2^{m}} < 2x_{k}$, we have $f_{k-1}\left(e + \frac{1}{2^{m}}\right) = 3$ by Lemma \ref{dwojkowowymierna}. If $x_{k-2}>e$ then we can find an $m$ so big that $e+\frac{1}{2^{m}} < x_{k-2}$. But then $f(e)=f_{k-1}(e)=3$, a contradiction. Therefore, $x_{k-2}=e$. However, then $f_{k-2}(e)=3$ and we can repeat the reasoning, and get $x_{k-3}=e$. 
\begin{align*}
f\left(e+\frac{1}{2^m}\right)\geq f_{k-3}\left(e+\frac{1}{2^m}\right)=f_{k}\left(e+\frac{1}{2^m}\right)+3f_{k}\left(\frac{1}{2^m}\right)=1+3\cdot 2 = 7.
\end{align*}

In both cases we obtained the required lower bounds.
\end{proof}
\end{theorem}

\begin{corollary}
$\{1,2,4,6\}$, $\{1,2,4,5,6\}\notin \mathcal{I}_1$.
\end{corollary}

\begin{lemma}\label{codolozycdoniedwojkowowymiernej}
Let $(x_n)$ be an interval-filling sequence such that there exists $k$ for which $x_n=\frac{1}{2^{n-k+1}}$ for $n\geq k$, $x_{k-1}\notin \mathcal{D}$, $x_{k-2}\geq 1$ and $x_n\geq 1+\sum_{i=n+1}^{k-2}x_i$ for every $n\leq k-3$. Denote the cardinal function of $(x_n)$ by $f$. Then $\max(f)\leq 4$. 
\begin{proof}
Let $g$ be the cardinal function of $(x_n)_{n\neq k-1}$. Then $\rng(g)=\{1,2\}$. Indeed, note that $(x_n)_{n<k-1}$ is quickly convergent, so each element of $\mathcal{A}((x_n)_{n<k-1})$ is obtained uniquely by a subsum of first $k-2$ terms, see \cite[Proposition 3.1]{cardfun}. Hence, if $d\in \mathcal{A}((x_n)_{n<k-1})+\mathcal{D}$ then $g(d)=2$ and if $x_n=1+\sum_{i=n+1}^{k-2}x_i$ then $g(x_n)=2$. The function $g$ evaluated on the other elements of $\mathcal{A}((x_n)_{n\neq k-1})$ is constantly equal to $1$. To conclude, observe that $f(x)=g(x)+g(x-e)\leq 2+2=4$. 
\end{proof}
\end{lemma}

\begin{theorem}\label{niemaczworkitosiedem}
Assume that $M\in \mathcal{I}_1$ is bounded, $4\notin M$ and $\max M\geq 5$. Then $\max M\geq 7$.
\end{theorem}
\begin{proof}
Parts of the following proof are very similar to the proof of Theorem \ref{conajmniejsiedem}. For this reason we omit some details.

Let $(x_n)$ be an interval-filling sequence with cardinal function $f$ with range $M$ satisfying all given assumptions.
Let  $x_n=\frac{1}{2^{n-k-1}}$ for all $n\geq k$ for some positive integer $k$, and $x_{k-1}< 2x_k$. We need to consider the following cases: 

\vspace{0.15cm}

\noindent
{\bf 1.} $x_{k-1}=d$ for some $d\in \mathcal{D}\cap [x_{k},2x_{k})$. Then $f_{k-1}(d+\frac{c}{2^p})=4$ for any integer $p$ such that $d+\frac{1}{2^p}<2x_{k}$. Hence $x_{k-2}=d$. But then $f_{k-2}(d)=4$ and we need to have $x_{k-3}=d$. Then $f_{k-3}(d+\frac{1}{2^p})=8$.

\vspace{0.15cm}

\noindent 
{\bf 2.} $x_{k-1}=e$ for some $e\in  (x_{k},2x_{k})\setminus \mathcal{D}$. Let $m\leq k-1$ be the minimal index such that $x_n\geq 1+\sum_{i=n+1}^{k-2}x_i$ holds for every $n\in \{m,m+1,\ldots,k-2\}$ (note that $x_n\leq 1+e+\sum_{i=n+1}^{k-2}x_i$ since $(x_n)$ is interval-filling).

From Lemma \ref{codolozycdoniedwojkowowymiernej} we know that the cardinal function $f_m$ of $\mathcal{A}((x_n)_{n\geq m})$ satisfies $\max(f_m)\leq 4$. Since $\max(f)\geq 5$ we get that $(x_n)_{n\geq m}$ is not equal to the whole sequence $(x_n)$, that is $m>1$.

Let us consider all of the possibilities for value of 
\begin{align*}
\alpha:=x_{m-1}< 1+\sum_{u=m}^{k-2}x_u.
\end{align*}
We use the notation $\mathcal{G}:=\mathcal{A}((x_n)_{n=m}^{k-2})$ and let $h$ be the cardinal function of $(x_n)_{n\geq m,n\neq k-1}$. Then clearly
\begin{align*}
f_{m-1}(x)=h(x)+h(x-e)+h(x-\alpha)+h(x-\alpha-e).
\end{align*}
Moreover, by Lemma \ref{codolozycdoniedwojkowowymiernej} we know that $h$ achieves only values $1$ and $2$. More precisely, $h(x)=2$ if and only if $x\in \mathcal{G}+\mathcal{D}$ or $x\in \mathcal{G}\cap(\mathcal{G}+c)$, and $h(x)=1$ otherwise. Furthermore, $\mathcal{G}\cap(\mathcal{G}+c)\subseteq (x_n)_{n=k-2}^{m-1}$ since by the construction of $((x_n)_{n=m}^{k-2})$ the distance between two different elements of $\mathcal{G}$ is at least equal to $1$.

Let us now consider four subcases.

\vspace{0.15cm}

\noindent
{\bf 2.1.} First we assume that $\alpha\notin \mathcal{G}+\mathcal{D}$ and $\alpha\notin \mathcal{G}+e+\mathcal{D}$. Since $\mathcal{G}\cap(\mathcal{G}+1)$ is finite, for any large enough $p$ we get $\alpha+\frac{1}{2^p}\notin \mathcal{G}\cap(\mathcal{G}+1)$, and then
\begin{align*}
f_{m-1}\left(\alpha+\frac{1}{2^p}\right)=h\left(\alpha+\frac{1}{2^p}\right)+h\left(\alpha-e+\frac{1}{2^p}\right)+h\left(\frac{1}{2^p}\right)+h\left(\frac{1}{2^p}-e\right)=1+1+2+0=4.
\end{align*}
Since $\alpha+\frac{1}{2^p}$ can be arbitrarily close to $\alpha$ and we need to avoid $f\left(\alpha+\frac{c}{2^p}\right)=f_{m-1}\left(\alpha+\frac{c}{2^p}\right)=4$, we need to have $x_{m-2}=\alpha$.

Since the sets $\mathcal{G}+\mathcal{D}$,  $\mathcal{G}+e+\mathcal{D}$, $\mathcal{G}\cap(\mathcal{G}+c)$ are at most countable, one can find arbitrarily small $t\notin \mathcal{D}$ such that $\alpha+t\notin (\mathcal{G}+\mathcal{D})\cup ( \mathcal{G}+e+\mathcal{D})\cup (\mathcal{G}\cap(\mathcal{G}+1))$. Then 
\begin{align*}
f_{m-2}(\alpha+t)=f_{m-1}(\alpha+t)+f_{m-1}(t)=h(\alpha+t)+h(\alpha+t-e)+h(t)+h(t-e)+1=1+1+1+0+1=4.
\end{align*}
Similarly as before, we conclude that $x_{m-3}=\alpha$. 
Thus
\begin{align*}
f_{m-3}\left(\alpha+\frac{1}{2^p}\right)=f_{m-1}\left(\alpha+\frac{1}{2^p}\right)+2f_{m-1}\left(\frac{1}{2^p}\right)+f_{m-1}\left(\frac{1}{2^p}-\alpha\right)=4+2\cdot 2+0=8.
\end{align*}
This means that $f\left(\alpha+\frac{c}{2^p}\right)\geq f_{m-3}\left(\alpha+\frac{c}{2^p}\right)=8\geq 7$.

\vspace{0.15cm}

\noindent
{\bf 2.2.} Let $\alpha\in \mathcal{G}+\mathcal{D}$ and $\alpha\notin \mathcal{G}+e+\mathcal{D}$. At first suppose that $\alpha-e\in \mathcal{G}\cap(\mathcal{G}+1)$. Then we get $\alpha=x_m+e$, so $\alpha=z+d$ for some $z\in \mathcal{G}$, $d\in \mathcal{D}$. Clearly $z\neq x_m$, so we obtain $1\leq x_m-z=d-e<1$, which is impossible. Hence $\alpha-e\notin \mathcal{G}\cap(\mathcal{G}+1)$. We have
\begin{align*}
f_{m-1}(\alpha)=h(\alpha)+h(\alpha-e)+h(0)+h(-e)=2+1+1+0=4,
\end{align*}
which implies that $x_{m-2}=\alpha$. Hence, for large enough $p$ we get 
\begin{align*}
f_{m-2}\left(\alpha+\frac{1}{2^p}\right) & =f_{m-1}\left(\alpha+\frac{1}{2^p}\right)+f_{m-1}\left(\frac{1}{2^p}\right)=f_{m-1}\left(\alpha+\frac{1}{2^p}\right)+2 \\
& =h\left(\alpha+\frac{1}{2^p}\right)+h\left(\alpha+\frac{1}{2^p}-e\right)+h\left(\frac{1}{2^p}\right)+h\left(\frac{1}{2^p}-e\right)+2=2+1+2+0+2=7,
\end{align*}
which gives $f\left(\alpha+\frac{1}{2^p}\right)\geq f_{m-2}\left(\alpha+\frac{1}{2^p}\right)=7$.

\vspace{0.15cm}

\noindent
{\bf 2.3.} In the third case we consider $\alpha\notin \mathcal{G}+\mathcal{D}$ and $\alpha\in \mathcal{G}+e+\mathcal{D}$. At first, we show that $\alpha\in \mathcal{G}\cap(\mathcal{G}+1)$ is not possible. Indeed, take $\alpha=z+e+d$, where $z\in \mathcal{G}$, $d\in \mathcal{D}$. Suppose that $\alpha\in \mathcal{G}+1$, that is, $\alpha=w+1$ for some $w\in \mathcal{G}$. Clearly $w\neq z$ and we obtain $1\leq w-z=e+d-1<1+1-1=1$. Hence, $\alpha\notin \mathcal{G}\cap(\mathcal{G}+1)$. We get 
\begin{align*}
f_{m-1}(\alpha)=h(\alpha)+h(\alpha-e)+h(0)+h(-e)=1+2+1+0=4,
\end{align*}
so $x_{m-2}=\alpha$. Then for large enough $p$ we get 
\begin{align*}
f_{m-2}\left(\alpha+\frac{1}{2^p}\right) & =f_{m-1}\left(\alpha+\frac{1}{2^p}\right)+f_{m-1}\left(\frac{1}{2^p}\right)=f_{m-1}\left(\alpha+\frac{1}{2^p}\right) \\
& =h\left(\alpha+\frac{1}{2^p}\right)+h\left(\alpha+\frac{1}{2^p}-e\right)+h\left(\frac{1}{2^p}\right)+h\left(\frac{1}{2^p}-e\right)+2 \\
& =1+2+2+0+2=7.
\end{align*}
This implies that $f(\alpha+\frac{1}{2^p})\geq f_{m-2}(\alpha+\frac{1}{2^p})=7$.

\vspace{0.15cm}

\noindent
{\bf 2.4.} In the last case we have $\alpha\in \mathcal{G}+\mathcal{D}$ and $\alpha\in \mathcal{G}+e+\mathcal{D}$.
Since $\alpha\in \mathcal{G}+\mathcal{D}$, we get $\alpha+e\in \mathcal{G}+e+\mathcal{D}$. Then for large enough $p$ we have $\alpha+\frac{1}{2^p}\in \mathcal{G}+\mathcal{D}$, so $\alpha+e+\frac{1}{2^p}\in \mathcal{G}+e+\mathcal{D}.$ Hence, 
\begin{align*}
f_{m-1}\left(\alpha+e+\frac{1}{2^p}\right)=h\left(\alpha+e+\frac{1}{2^p}\right)+h\left(\alpha+\frac{1}{2^p}\right)+h\left(e+\frac{1}{2^p}\right)+h\left(\frac{1}{2^p}\right)=2+2+1+2=7.
\end{align*}

In all the cases we obtained $\max M\geq 7$ as desired.
\end{proof}

\begin{corollary}
$\{1,2,3,5\}$, $\{1,2,3,5,6\}\notin \mathcal{I}_1$.
\end{corollary}

\section{Remarks about $\mathcal{F}$ and $\mathcal{C}$}

In this section we study families $\mathcal{F}$ and $\mathcal{C}$. At first, let us focus on $\mathcal{F}$. Before we proceed, we show that all the sets in $\mathcal{F}$ can be obtained as ranges of cardinal functions of sequences consisting of positive integers only.

\begin{theorem}\label{ThmNaturalneStarcza}
Suppose that $M\in \mathcal{F}$. Then there exist positive integers $n_{1},\ldots ,n_{k}$ such that $A$ is the range of the cardinal function of $(n_{1},\ldots ,n_{k})$.
\end{theorem}
\begin{proof}
In this proof we will denote by $f_{\mathbf{y}}$ the cardinal function of a sequence $\mathbf{y}$.

Let $M$ be the set from the statement. Let $\mathbf{x}$ be a finite sequence of real numbers such that $\rng (f_{\mathbf{x}})=M$. 
At first we will show that there exists a finite sequence $\mathbf{r}$ of rational numbers such that $\rng (f_{\mathbf{r}})=\rng (f_{\mathbf{r}})$. Let us consider the linear space $\mathcal{H}$ over $\mathbb{Q}$ spanned by $\mathbf{x}$. Let $B=\{b_{1},\ldots ,b_{s}\}$ be a basis of $\mathcal{H}$. Then every element $h\in\mathcal{H}$ (and in particular every element of $\mathbf{x}$) can be expressed as
\begin{align*}
h = \sum_{j=1}^{s}z_{h,j}b_{j},
\end{align*}
where all the numbers $z_{h,j}$ are rational.

Let $K$ be large enough (to be specified later), and consider the following map:
\begin{align*}
\varphi : \mathcal{H}\ni \sum_{j=1}^{s}z_{h,j}b_{j} \mapsto \sum_{j=1}^{s}z_{h,j}K^{j} \in \mathbb{Q}.
\end{align*}

Let $\mathbf{r}:=\varphi (\mathbf{x})$. Then of course $\mathbf{r}\subseteq \mathbb{Q}$. We need to prove that if $K$ is large enough, $\varphi (\mathcal{A}(\mathbf{x}))=\mathcal{A}(\mathbf{r})$. That is, we want to have that if $x_{1},\ldots ,x_{n}\in\mathbf{x}$ then
\begin{align*}
\varphi\left(\sum_{i=1}^{n} x_{i}\right) = \sum_{i=1}^{n} \varphi (x_{i}).
\end{align*}
The above is true if we take for example $K:= 2+ 2\sum_{j=1}^{s} \sum_{x\in\mathbf{x}} |z_{x,j}|$. Moreover, one can check that with the same choice of $K$, the function $\varphi\mid_{\mathbf{x}}:\mathbf{x}\to \mathbf{r}$ is an injection, so a bijection because $\mathbf{x}$ and $\mathbf{r}$ are finite.

It remains to show that $\rng (f_{\mathbf{x}})=\rng(f_{\mathbf{r}})$. Let us consider two sequences $(\alpha_{1}',\ldots ,\alpha_{n}')$ and $(\beta_{1}',\ldots ,\beta_{m}')$ of elements of $\mathbf{r}$ such that
\begin{align}\label{EqNaturalneStarcza1}
\sum_{i=0}^{n}\alpha_{i}' = \sum_{i=0}^{m}\beta_{i}'.
\end{align}
For every $i$ and $j$ there exist unique elements $\alpha_{i}$ and $\beta_{j}$ from $\mathbf{x}$ such that $\alpha_{i}' = \varphi (\alpha_{i})$ and $\beta_{j}'=\varphi (\beta_{j})$. Then equality \eqref{EqNaturalneStarcza1} is equivalent to
\begin{align*}
\sum_{i=1}^{n}\sum_{j=1}^{s}z_{\alpha_{i},j} K^{j} = \sum_{i=1}^{m}\sum_{j=1}^{s}z_{\beta_{i},j} K^{j},
\end{align*}
that is,
\begin{align*}
\sum_{j=1}^{s}\left( \sum_{i=1}^{n} z_{\alpha_{i},j}\right) K^{j} = \sum_{j=1}^{s} \left( \sum_{i=1}^{m} z_{\beta_{i},j} \right) K^{j}.
\end{align*}
The value of $K$ is large enough that the above condition is equivalent to
\begin{align*}
\sum_{i=1}^{n} z_{\alpha_{i},j} = \sum_{i=1}^{m} z_{\beta_{i},j}
\end{align*}
for every $j$. This condition is further equivalent to
\begin{align*}
\sum_{j=1}^{s}\left( \sum_{i=1}^{n} z_{\alpha_{i},j}\right) b_{j} = \sum_{j=1}^{s} \left( \sum_{i=1}^{m} z_{\beta_{i},j} \right) b_{j},
\end{align*}
which after simple computations is the same, as
\begin{align*}
\sum_{i=0}^{n}\alpha_{i} = \sum_{i=0}^{m}\beta_{i}.
\end{align*}
In other words, we have just proved that
\begin{align*}
\sum_{i=0}^{n}\alpha_{i}' = \sum_{i=0}^{m}\beta_{i}' \hspace{1.5cm} \textrm{ if and only if } \hspace{1.5cm} \sum_{i=0}^{n}\alpha_{i} = \sum_{i=0}^{m}\beta_{i}.
\end{align*}
In particular, he ranges of the cardinal functions of sequences $\mathbf{x}$ and $\mathbf{r}$ are the same. 

Let $c$ be equal to the product of the denominators of the elements of $\mathbf{r}$. Then $\rng (f_{\mathbf{r}}) = \rng (f_{c\mathbf{r}})$, so we can assume without lose of generality, that $\mathbf{r}$ consists of integers only. Then by \cite[Section 2]{cardfun} we know that there exists a sequence $\mathbf{n}$ of positive integers such that $\rng (f_{\mathbf{r}}) = \rng (f_{\mathbf{n}})$. This finishes the proof.

\end{proof}

Note here, that it is proved in \cite{cardfun} that the following sets: 
\begin{align*}
& \{1\},\ \{1,2\},\ \{1,3\},\  \{1,2,3\},\ \{1,2,4\},\ \{1,3,4\},\ \{1,2,3,4\},\ \{1,2,3,5\},\\ & 
\{1,2,4,5\},\ \{1,3,6\},\ \{1,4,6\},\ \{1,2,3,6\},\ \{1,2,4,6\},\ \{1,3,4,6\},\ \{1,4,5,6\}, \\
& \{1,2,3,4,6\},\ \{1,2,4,5,6\},\ \{1,3,4,5,6\},\ \{1,2,3,4,5,6\}
\end{align*}
are in $\mathcal{F}$.

Now let us present results that provide some necessary conditions for a set $M$ to be long to $\mathcal{F}$. We apply our results to the case of ranges $M$ of cardinal functions satisfying $\max M\leq 6$. 



\begin{theorem}\label{Thm1,A}
Let $k\geq 4$ be an integer and $A\subseteq \left\{k,k+1,k+2,\ldots, {k \choose \lfloor k/2 \rfloor}-1\right\}$. Then $\{1\}\cup A\notin \mathcal{F}$.
\end{theorem}
\begin{proof}
Suppose that $A\subseteq \{k,k+1,k+2,\ldots, {k \choose \lfloor k/2 \rfloor}-1\}$ is such that  $\{1\}\cup A$ the range of the cardinal function $f$ of some finite sequence $\textbf{x}=(x_n)_{n=1}^{N}$. Let $y\in\mathcal{A}(\textbf{x})$ be the smallest number that can be represented in more than one way as a sum of elements of $\textbf{x}$. In other words, $y$ is the smallest such that $f(y)>1$, and from the fact, that $\min A= k$, we have $f(y)\geq k$. Hence, there exist distinct sets $A_{1},A_{2},\ldots ,A_{k}\subseteq \{1,\ldots ,N\}$ such that $y=\sum_{n\in A_i}x_n$ for each $i\in\{1,\ldots,k\}$.

Suppose that $A_p\cap A_q\neq\emptyset$ for some $p\neq q$, $p,q\in\{1,\ldots,k\}$. Let $x=\sum_{n\in A_p\cap A_q}x_n$. But then $y-x<y$ and 
\begin{align*}
y-x = \sum_{n\in A_p\setminus A_q}x_n = \sum_{n\in A_q\setminus A_p}x_n,
\end{align*}
so $f(y-x)>1$, contradicting the minimality of $y$. Therefore $A_p\cap A_q=\emptyset$ for every $p\neq q$, $p,q\in\{1,\ldots,k\}$.

It follows from the above discussion, that for every $j\in\{1,\ldots ,k\}$, number $jy$ has at least $\binom{k}{j}$ different expressions as a sum of elements of $\textbf{x}$. Indeed, every expression of the form
\begin{align*}
jy = \sum_{n\in A_{i_{1}}\cup \cdots \cup A_{i_{j}}} x_{n}
\end{align*} 
for some $1\leq i_{1}<\cdots <i_{j}\leq k$ is valid. Thus $f(jy)\geq \binom{k}{j}$, which is maximised by $j=\lfloor\frac{k}{2}\rfloor$, so 
\begin{align*}
f\left(\left\lfloor\frac{k}{2}\right\rfloor y\right) \geq \binom{k}{\lfloor\frac{k}{2}\rfloor}.
\end{align*}
However, we assume that  $\max A<{k \choose \lfloor k/2 \rfloor}$, a contradiction.
\end{proof}

\begin{corollary}
Suppose that one of the following holds:
\begin{enumerate}
\item $A\subseteq \{1,4,5\}$,
\item $A\subseteq \{1,5,6,7,8,9\}$,
\item $A\subseteq \{1,6,7,8,\ldots,19\}$.
\end{enumerate}
Then $A\notin \mathcal{F}$.
\end{corollary}
\begin{proof}
This is an immediate consequence of Theorem \ref{Thm1,A}.
\end{proof}

Note that the number $\binom{k}{k/2}-1$ in the statement of Theorem \ref{Thm1,A} cannot be improved for all $k$ because for example $\{1,4,6\}$, $\{1,5,10\}$, $\{1,6,15,20\}\in \mathcal{F}$.

Let us move to the set $\mathcal{C}$. In the paper \cite{cardfun} it is proved that in fact $\mathcal{F}\subseteq \mathcal{C}$ (the proof uses the fact that $\{1\}\in \mathcal{C}$ and a method similar to the one used in the proof of Lemma \ref{dolozskonczony}). Unfortunately, no other method of constructing Cantor set with the prescribed range of cardinal function is known.

Let $\Fin$ denote the ideal of all finite subsets of $\mathbb{N}$. Clearly $\mathcal{F}\subseteq \mathcal{C}$ and $\mathcal{F}\subseteq \Fin$. The following question seems to be very interesting.

\begin{problem}
Does the equality $\mathcal{F}=\mathcal{C}\cap \Fin$ hold?
\end{problem}

\section{Remarks about $\Cv$}
The case of Cantorval is the least studied for cardinal functions. 
Note that the achievement sets was mostly studied for a multigeometric sequences \cite{BBFS}, \cite{BBGS}, \cite{BFS}, that is of the form 
$$(x_n)=(a_1,a_2,\ldots,a_m;q):=(a_1q,a_2q,\ldots,a_mq,a_1q^2, a_2q^2,\ldots,a_mq^2,a_1q^3,\ldots).$$
If we denote
\begin{align*}
\Sigma:=\mathcal{A}(a_1,\ldots,a_m)=  \left\{\ \sum_{n=1}^{m}\varepsilon_na_n \ \Bigg|\ (\varepsilon_n)\in\{0,1\}^{m}\ \right\}
\end{align*}
then
\begin{align*}
\mathcal{A}(a_1,a_2,\ldots,a_m;q)=\left\{\ \sum_{n=1}^{\infty}x_nq^n\ \Bigg| \ (x_n)\in\Sigma^{\infty}\ \right\}.
\end{align*}

The only known result regarding ranges of cardinal function in $\Cv$ comes from \cite{BPW}, where the Guthrie--Nymann Cantorval $\GN:=\mathcal{A}\left(3,2;\frac{1}{4}\right)$ was considered. The result \cite[Theorem 6.1]{BPW} implies that $\{1,2\}\in\Cv$. More precisely, a point $x\in \GN$ has two representations
\begin{align*}
x=\sum_{n=1}^{\infty}\frac{a_n}{4^n}=\sum_{n=1}^{\infty}\frac{b_n}{4^n}
\end{align*}
if and only if there exists the finite or infinite sequence $n_0<n_1<\ldots$ such that 
\begin{enumerate}
\item $a_k=b_k$ for $0<k<n_0$;
\item $a_{n_0}=2$ and $b_{n_0}=3$;
\item $a_{n_k}=5$ and $b_{n_k}=0$ for odd $k$;
\item $a_{n_k}=0$ and $b_{n_k}=5$ for even $k>0$;
\item  $a_{i}\in\{3,5\}$ and $a_i-b_{i}=3$, whenever $n_{2k}<i<n_{2k+1}$;
\item  $a_{i}\in\{0,2\}$ and $b_i-a_{i}=3$, whenever $n_{2k+1}<i<n_{2k+2}$.
\end{enumerate} 
Otherwise, a point has a unique representation.

Our Lemma \ref{dolozskonczony} allows us to immediately produce more examples of sets from $\Cv$. In particular, we get the following sets with the maximal element at most $6$.

\begin{proposition}
$\{1,2,4\}$, $\{1,2,3,6\}$, $\{1,2,3,4,6\}\in \Cv$.
\end{proposition}
\begin{proof}
As we mentioned, $\{1,2\}\in \Cv$. Hence, Lemma \ref{dolozskonczony} yields $F\cdot \{1,2\}\in \Cv$ for every $F\in\mathcal{F}$. We get the statement by choosing $F=\{1,2\}$, $\{1,3\}$ and $\{1,2,3\}$.
\end{proof}

Note that if we want to construct $M\in \Cv\cap \Fin$ from a multigeometric sequence, the shape of the set $\Sigma$ is crucial. More precisely, the function 
\begin{align*}
\{0,1\}^{m}\ni (\varepsilon_n)\rightarrow \sum_{n=1}^{m}\varepsilon_na_n\in \Sigma
\end{align*}
needs to be a bijection onto $\Sigma$. Otherwise $\mathfrak{c}\in M$, and thus $M\notin\Fin$.


In order to find other sets in $\Cv$, we will need to perform more computations.

\begin{proposition}
$\{1,2,3,4\}$, $\{1,2,3,4,5,6\}\in \Cv$.
\end{proposition}
\begin{proof}
We consider both cases separately. We construct required sequences by adding finitely many terms to the sequence $\left(3,2;\frac{1}{4}\right)$. Hence, their achievement sets are Cantorvals.

In order to achieve $\{1,2,3,4\}$, let us consider $\mathbf{x} := \frac{3}{4}\uplus \left(3,2;\frac{1}{4}\right)$. Let $f$ be the cardinal function of the sequence $\left(3,2;\frac{1}{4}\right)$, and $g$ be the cardinal function of $\mathbf{x}$. Then
\begin{align*}
g(x)=f(x)+f\left(x-\frac{3}{4}\right)\leq 2+2=4.
\end{align*}

In order to conclude the proof in this case, observe that
\begin{itemize}
\item $g(0)=1$,
\item $g\left(\frac{3}{16}\right)=f\left(\frac{3}{16}\right)=2$,
\item $g\left(\frac{3}{4}\right)=f\left(\frac{3}{4}\right)+f(0)=2+1=3$,
\item $g\left(\frac{3}{4}+\frac{3}{16}\right)=f\left(\frac{3}{4}+\frac{3}{16}\right)+f\left(\frac{3}{16}\right)=2+2=4$.
\end{itemize}
Hence, $\rng (f)=\{1,2,3,4\}$.


The case of $\{1,2,3,4,5,6\}$ is more complicated. Consider $\mathbf{x}:= \frac{11}{12}\uplus \left(\frac{11}{12}\uplus \left(3,2;\frac{1}{4}\right)\right)$. By Lemma \ref{mniejnizszesc} we get that $\mathcal{A}(\mathbf{x})\subseteq \{1,2,3,4,5,6\}$.

Let $h$, $g$, $f$ denote the cardinal functions of $\mathbf{x}$, $\frac{11}{12}\uplus \left(3,2;\frac{1}{4}\right)$ and $\left(3,2;\frac{1}{4}\right)$, respectively. In particular, $g(x)=f(x)+f\left(x-\frac{11}{12}\right)$. Therefore,
\begin{align*}
h(x) & =g(x)+g\left(x-\frac{11}{12}\right)=f(x)+f\left(x-\frac{11}{12}\right)+f\left(x-\frac{11}{12}\right)+f\left(x-\frac{11}{12}-\frac{11}{12}\right) \\
& =f(x)+2f\left(x-\frac{11}{12}\right)+f\left(x-\frac{11}{6}\right).
\end{align*}

We have $\rng(f)=\{1,2\}$. Since $h(x)=f(x)$ for all $x\in \GN\cap[0,x_1)$, we get $1,2\in \rng(h)$.

Note that $\frac{11}{12}\in [\frac{2}{3},1]\subseteq \GN$. We have
\begin{align*}
\frac{11}{12}=\frac{3}{4}+\sum_{n=2}^{\infty}\frac{2}{4^n}=\frac{2}{4}+\sum_{n=2}^{\infty}\frac{3+2}{4^n},
\end{align*}
so $f(\frac{11}{12})=2$. Therefore, the following equalities hold:
\begin{align*}
h\left(\frac{11}{6}\right) & = f\left(\frac{11}{6}\right)+2\cdot f\left(\frac{11}{12}\right)+f(0)=0+2\cdot 2+1=5, \\
h\left(\frac{11}{12}\right) & =f\left(\frac{11}{12}\right)+2\cdot f(0)+f\left(-\frac{11}{12}\right)=2+2\cdot 1+0=4.
\end{align*}
Hence, $4,5\in \rng (h)$. 

Now observe that
\begin{align*}
\frac{7}{6}=\frac{3}{4}+\sum_{n=2}^{\infty}\frac{5}{4^n}\in \GN.
\end{align*}
By \cite[Theorem 6.1]{BPW} we obtain $f\left(\frac{7}{6}\right)=1$. Moreover, 
\begin{align*}
\frac{7}{6}-\frac{11}{12}=\frac{1}{4}=\sum_{n=2}^{\infty}\frac{3}{4^n},
\end{align*}
so $f\left(\frac{1}{4}\right)=1$. Therefore, we have 
\begin{align*}
h\left(\frac{7}{6}\right)=f\left(\frac{7}{6}\right)+2\cdot f\left(\frac{1}{4}\right)+f\left(-\frac{2}{3}\right)=1+2\cdot 1 + 0 = 3.
\end{align*} 

In order to finish the calculations, we need to show that $6\in \rng (h)$. 
We have 
\begin{align*}
\frac{3}{64}=\frac{2}{64}+\sum_{n=4}^{\infty}\frac{3}{4^n},
\end{align*}
so $f\left(\frac{3}{64}\right)=2$. What is more, 
\begin{align*}
\frac{185}{192}=\frac{11}{12}+\frac{3}{64}=\frac{3}{4}+\frac{2}{16}+\frac{5}{64}+\sum_{n=4}^{\infty}\frac{2}{4^n}=\frac{3}{4}+\frac{3}{16}+\sum_{n=4}^{\infty}\frac{5}{4^n},
\end{align*}
which implies that $f\left(\frac{185}{192}\right)=2$. Therefore,
\begin{align*}
h\left(\frac{361}{192}\right)=h\left(\frac{11}{6}+\frac{3}{64}\right)=f\left(\frac{11}{6}+\frac{3}{64}\right)+2\cdot f\left(\frac{11}{12}+\frac{3}{64}\right)+f\left(\frac{3}{64}\right)=0+2\cdot 2+2=6.
\end{align*}
We obtained $\rng (h)=\{1,2,3,4,5,6\}$ and the proof is finished.
\end{proof}

\newpage
\section{Summary}
In this chapter we sum up  known results connected with ranges $M$ with $\max M\leq 6$ and for $M=\{1,3,5,7\}$. All of them comes from our paper and the paper \cite{cardfun}. We use symbol \ding{52}  to denote that it is possible to construct a sequence with the given type of achievement set and range of cardinal function, while symbol \ding{56} means that it has been proved that it can not be constructed. The blank spaces are left for unknown cases.
\\\begin{center}
\begin{tabular}{|c||c|c|c|c|c|c|}
\hline
\hline
range $\backslash$ type  & $\mathcal{I}_1$ & $\mathcal{I}$ & $\mathcal{F}$ & $\mathcal{C}$ & $\Cv$ & $\mathcal{R}$\\
\hline
\hline
$\{1\}$ & \ding{56}& \ding{56} & \ding{52} & \ding{52} & \ding{56} & \ding{52}\\
\hline
$\{1,2\}$ & \ding{52} & \ding{52} & \ding{52} & \ding{52} & \ding{52} & \ding{52}\\
\hline
$\{1,3\}$ & \ding{56} & \ding{56} & \ding{52} & \ding{52} & & \ding{52}\\
\hline
$\{1,2,3\}$ & \ding{52} & \ding{52} & \ding{52} & \ding{52}& & \ding{52}\\
\hline
$\{1,4\}$ & \ding{56} & \ding{56} & \ding{56} &  & & \\
\hline
$\{1,2,4\}$ & \ding{56} & \ding{52} & \ding{52} & \ding{52} & \ding{52} & \ding{52}\\
\hline
$\{1,3,4\}$ & \ding{56}  & \ding{56} & \ding{52} & \ding{52} & & \ding{52}\\
\hline
$\{1,2,3,4\}$  & \ding{52}  & \ding{52} & \ding{52} & \ding{52} & \ding{52} & \ding{52}\\
\hline
$\{1,5\}$ & \ding{56} & \ding{56} & \ding{56} &  & & \\
\hline
$\{1,2,5\}$  & \ding{56} &  &  &  & & \\
\hline
$\{1,3,5\}$  & \ding{56} & \ding{56} &  &  & & \\
\hline
$\{1,4,5\}$   & \ding{56} & \ding{56} &  \ding{56} &  & & \\
\hline
$\{1,2,3,5\}$   & \ding{56}  &  &  \ding{52}  & \ding{52}  & & \ding{52} \\
\hline
$\{1,2,4,5\}$  & \ding{56} &  &  \ding{52}  & \ding{52}  & & \ding{52} \\
\hline
$\{1,3,4,5\}$   & \ding{56} & \ding{56} &  &  & & \\
\hline
$\{1,2,3,4,5\}$   & \ding{52}  & \ding{52}  &  &  & &  \ding{52} \\
\hline
$\{1,6\}$   &  \ding{56} & \ding{56} &   \ding{56} &  & & \\
\hline
$\{1,2,6\}$   &  \ding{56} &  &  &  & & \\
\hline
$\{1,3,6\}$   & \ding{56} & \ding{56} & \ding{52} & \ding{52} & & \ding{52} \\
\hline
$\{1,4,6\}$   & \ding{56} & \ding{56} & \ding{52} & \ding{52} & & \ding{52} \\
\hline
$\{1,5,6\}$   & \ding{56} & \ding{56} & \ding{56} &  & & \\
\hline
$\{1,2,3,6\}$  & \ding{56} &  \ding{52} & \ding{52} & \ding{52} & \ding{52} & \ding{52} \\
\hline
$\{1,2,4,6\}$  & \ding{56} &  & \ding{52} & \ding{52} &  & \ding{52} \\
\hline
$\{1,2,5,6\}$   &  \ding{56} &  &  &  & & \\
\hline
$\{1,3,4,6\}$   & \ding{56} & \ding{56} & \ding{52} & \ding{52} & & \ding{52} \\
\hline
$\{1,3,5,6\}$   & \ding{56} & \ding{56} &  &  & & \\
\hline
$\{1,4,5,6\}$   & \ding{56} & \ding{56} & \ding{52} & \ding{52} & & \ding{52} \\
\hline
$\{1,2,3,4,6\}$   & \ding{52}  &   \ding{52} & \ding{52} & \ding{52} & \ding{52} & \ding{52} \\
\hline
$\{1,2,3,5,6\}$   & \ding{56} & &  &  & & \\
\hline
$\{1,2,4,5,6\}$   & \ding{56} &  & \ding{52}  & \ding{52} & & \ding{52}\\
\hline
$\{1,3,4,5,6\}$   & \ding{56} & \ding{56} & \ding{52} & \ding{52} & & \ding{52}\\
\hline
$\{1,2,3,4,5,6\}$   &  \ding{52} &  \ding{52} & \ding{52} & \ding{52} & \ding{52} & \ding{52}\\
\hline
$\{1,3,5,7\}$   &  \ding{56} &  \ding{56} & \ding{52} & \ding{52} &  & \ding{52} \\
\hline
\end{tabular}
\end{center}

\section*{Acknowledgements}
The authors are very grateful to Bartosz Sobolewski for many valuable comments to the previous version of the paper.

\end{document}